\documentclass[a4paper,10pt]{amsart}
\usepackage{amsmath,amsthm,amsfonts,amssymb,calrsfs}

\usepackage{enumitem}
\usepackage{tikz-cd, tikz}
\usepackage{graphicx}
\usepackage{yfonts}
\usepackage{xcolor}

\usepackage[hyperindex=true, colorlinks=true, linkcolor=blue, filecolor=black,
citecolor=blue, urlcolor=blue, pagebackref=true, bookmarks=true,
bookmarksnumbered=true, bookmarksopenlevel=2, pdfborder={0 0 0},
pdftitle={Toric sheaves, stability and fibrations}]{hyperref}

\theoremstyle{plain}{
    \newtheorem{theorem}{Theorem}[section]
    \newtheorem{lem}[theorem]{Lemma}
    \newtheorem{cor}[theorem]{Corollary}
    \newtheorem{prop}[theorem]{Proposition}
    
}

\theoremstyle{definition}{
    \newtheorem{defn}[theorem]{Definition}
    \newtheorem{defn-thm}[theorem]{Definition-Theorem}
    \newtheorem{example}[theorem]{Example}
    \newtheorem{notation}[theorem]{Notation}
}

\theoremstyle{remark}{
    \newtheorem{rem}[theorem]{Remark}

}

\def\Id{\mathrm{Id}}

\def\cF{\mathcal{F}}
\def\cE{\mathcal{E}}
\def\cO{\mathcal{O}}

\def\cN{\mathcal{N}}
\def\cG{\mathcal{G}}

\def\cT{\mathcal{T}}

\def\fK{\mathfrak{K}}
\def\fE{\mathfrak{E}}
\def\fF{\mathfrak{F}}

\def\K{\mathbb{K}}
\def\R{\mathbb{R}}
\def\Q{\mathbb{Q}}
\def\Z{\mathbb{Z}}

\def\P{\mathbb{P}}

\def\C{\mathbb{C}}

\def\E{\mathbb{E}}

\def\ep{\varepsilon}

\def\>{\rangle}

\def\<{\langle}
\def\>{\rangle}

\def\Tor{\mathrm{Tor}}
\def\Hom{\mathrm{Hom}}
\def\Spec{\mathrm{Spec}}
\def\rk{\mathrm{rk}}
\def\deg{\mathrm{deg}}
\def\Span{\mathrm{Span}}
\def\Div{\mathrm{Div}}           
\def\WDiv{\mathrm{WDiv}}                    
\def\Gr{\mathrm{Gr}} 

\def\Cone{\mathrm{Cone}}

\def\l{\ell}

\newcommand{\disp}{\displaystyle}

\newcommand{\eqv}{\longleftrightarrow}

\newcommand{\imp}{\longrightarrow}

\DeclareMathOperator{\ddiv}{div}
\DeclareMathOperator{\mult}{mult}
\DeclareMathOperator{\Ker}{Ker}
\DeclareMathOperator{\Amp}{Amp}

\DeclareMathOperator{\Pic}{Pic}
\DeclareMathOperator{\Bl}{Bl}
\DeclareMathOperator{\pr}{pr}
\begin{document}

\title[Toric sheaves, stability and fibrations]{Toric sheaves, stability and fibrations}

\address{Univ Brest, UMR CNRS 6205, Laboratoire de Mathématiques de Bretagne
Atlantique, France}

\author[A. Napame]{Achim NAPAME}
\email{achim.napame@univ-brest.fr}

\author[C. Tipler]{Carl TIPLER}
\email{carl.tipler@univ-brest.fr}


\begin{abstract}
For an equivariant reflexive sheaf over a polarised toric variety, we study slope stability of its reflexive pullback along a toric fibration. Examples of such fibrations include equivariant blow-ups and toric locally trivial fibrations. We show that stability (resp. unstability) is preserved under such pullbacks for so-called adiabatic polarisations. In the strictly semistable situation, under locally freeness assumptions, we provide a necessary and sufficient condition on the graded object to ensure stability of the pulled back sheaf. As applications, we provide various stable perturbations of semistable tangent sheaves, either by changing the polarisation, or by blowing-up a subvariety. Finally, our results apply uniformly in specific flat families and induce injective maps between the associated moduli spaces. 
\end{abstract}

\maketitle

\section{Introduction}
The study of moduli spaces of torsion-free coherent sheaves on a given variety is a fundamental problem in algebraic geometry. The construction of a quasi-projective structure on the moduli space can be achieved by considering stable sheaves \cite{HuLe}. In this paper, we will be interested in slope stability as introduced by Mumford and Takemoto \cite{Tak72}. Stable reflexive sheaves are of particular interest, given their close relation to stable vector bundles \cite{Har80}. Being tightly linked to the geometry of the ambiant variety, it is natural to investigate how they behave with respect to natural maps such as pullbacks. In this direction, a fundamental result of Mehta and Ramanathan \cite{MeRa} asserts that the restriction of a slope (semi)stable torsion-free sheaf to a generic complete intersection of high degree remains slope (semi)stable. In this paper, we address the problem of pulling-back a (semi)stable reflexive sheaf along a fibration, in the equivariant context of toric geometry.

Consider a toric fibration $\pi : X'\to X$ between $\Q$-factorial projective toric varieties defined over the complex numbers (see Section \ref{sec:background} for precise definitions). We will denote by $T$ the torus of $X$. Assume that $L$ is an ample divisor on $X$ and $L'$ is a relatively ample divisor on $X'$. Then, for $\ep \in \Q$ small enough, $L_\ep:=\pi^*L+\ep L'$ defines an ample $\Q$-divisor on $X'$. Following the terminology used in differential geometry, we will call the associated polarisation {\it adiabatic}. For a given torsion-free sheaf $\cE$ on $X$, we will denote by $\mu_L(\cE):=\deg_L(\cE)/\rk(\cE)$ its slope with respect to $L$, and call $\cE$ a {\it stable} (resp. {\it semistable}) sheaf with respect to $L$ if for all coherent subsheaves $\cF\subseteq\cE$ with strictly smaller rank, we have $\mu_L(\cF)<\mu_L(\cE)$ (resp. $\mu_L(\cF)\leq\mu_L(\cE)$). If for one subsheaf $\cF\subseteq\cE$ we have $\mu_L(\cF)>\mu_L(\cE)$, $\cE$ is called {\it unstable}. Then, our first result is the following:

\begin{theorem}
\label{theo:stablecaseintro}
Let $\cE$ be a $T$-equivariant stable reflexive sheaf on $(X,L)$. Then there is $\ep_0>0$
such that for all $\ep\in (0,\ep_0)\cap \Q$, the reflexive pullback $(\pi^*\cE)^{\vee \vee}$
is stable on $(X',L_\ep)$.
\end{theorem}

This result relies essentially on the fact that in the torus equivariant context, it is
enough to test slope inequalities for equivariant and reflexive saturated subsheaves
\cite{Koo11,HNS19}. In general, the slope of a pulled back sheaf on $X'$ with respect to
$L_\ep$ admits an expansion in $\ep$, with the coefficient of the smallest exponent given
by the slope of the initial sheaf on $X$ with respect to $L$. Hence, the following is
straightforward:

\begin{prop}
\label{prop:unstablecaseintro}
Let $\cE$ be a $T$-equivariant unstable reflexive sheaf on $(X,L)$. Then there is $\ep_0>0$
such that for all $\ep\in (0,\ep_0)\cap \Q$, the reflexive pullback $(\pi^*\cE)^{\vee \vee}$
is unstable on $(X',L_\ep)$.
\end{prop}

Our main result deals with the more delicate strictly semistable situation. Let $\cE$ be a
strictly semistable torsion-free sheaf on $(X,L)$. It then admits a Jordan-H\"older
filtration
$$
0=\cE_1 \subseteq \cE_2 \subseteq \ldots \subseteq \cE_\ell=\cE
$$
by slope semistable coherent subsheaves with stable quotients of same slope as $\cE$ \cite{HuLe}. The reflexive pullbacks of the $\cE_i$'s form natural candidates to test for stability of the reflexive pullback of $\cE$ on $(X',L_\ep)$. In fact, we will see shortly that if $\cE$ and $\Gr(\cE):=\bigoplus_{i=1}^{\ell-1}\cE_{i+1}/\cE_i $ are locally free, it is actually enough to compare slopes with these sheaves. In order to state our result, we will introduce some notations.
Let $\mathfrak{E}$ be the set of equivariant and saturated reflexive subsheaves
$\cF \subseteq \cE$ arising in a Jordan-H\"older filtration of $\cE$.
For two coherent sheaves $\cF_1$ and $\cF_2$ on $X'$, we will write
$\mu_0(\cF_1)<\mu_0(\cF_2)$ (resp. $\mu_0(\cF_1)\leq \mu_0(\cF_2)$ or
$\mu_0(\cF_1)=\mu_0(\cF_2)$) when the coefficient of the smallest exponent in the expansion
in $\ep$ of $\mu_{L_\ep}(\cF_2)-\mu_{L_\ep}(\cF_1)$ is strictly positive (resp. greater or
equal to zero or equal to zero). Recall that a locally free semistable sheaf is called
{\it sufficiently smooth} if its graded object is locally free.

\begin{theorem}
\label{theo:semistablecaseintro}
Let $\cE$ be a $T$-equivariant locally free and sufficiently smooth strictly semistable sheaf
on $(X,L)$. Then there is $\ep_0>0$ such that for all $\ep\in (0,\ep_0)\cap \Q$, the
reflexive pullback $\cE':=(\pi^*\cE)^{\vee \vee}$ on $(X', L_\ep)$ is:
 \begin{itemize}
  \item[(i)] stable iff for all $\cF\in\fE$, $\mu_{0} ((\pi^\ast \cF)^{\vee \vee})< \mu_{0} (\cE')$,
  \item[(ii)] strictly semistable iff for all $\cF\in\fE$, $\mu_{0} ((\pi^\ast \cF)^{\vee \vee})\leq \mu_{0} (\cE')$ with at least one equality,
  \item[(iii)] unstable iff there is one $\cF\in\fE$ with $\mu_{0} ((\pi^\ast \cF)^{\vee \vee})> \mu_{0} (\cE')$.
 \end{itemize}
\end{theorem}

This theorem should be compared to \cite[Theorem 1.4]{SekTip22}, where a similar result is obtained, in a non necessarily toric setting, for pullbacks of strictly semistable vector bundles on holomorphic submersions.
The approach to the problem in \cite{SekTip22} is fairly different, with differential geometric techniques, and requires some additional technical assumptions on $\Gr(\cE)$. Working in the toric setting, by mean of combinatorial and algebraic methods, we are able to extend the results from \cite{SekTip22} to more general fibrations, allowing singularities on $X$ and $X'$, and fibers with multiple irreducible components. 
\begin{rem}
We will see in Section \ref{sec:loctrivialfibration} that if $\pi$ is a locally trivial fibration, our assumption on $\cE$ and $\Gr(\cE)$ to be locally free in Theorem \ref{theo:semistablecaseintro} is not necessary. In that situation, we can remove all the technical hypothesis that were required in \cite[Theorem 1.4]{SekTip22}.
\end{rem}
\begin{rem}
  In \cite{CT22}, a closely related problem is considered. Let $(X',L')$ be a normal toric variety and $X=X'//G$ a GIT quotient under the action of a generic subtorus $G$ of the torus of $X'$. Denote by $\iota : X^s \to X'$ the  inclusion of the stable locus under this action, and $\pi : X^s \to X$ the quotient map. Then a combinatorial condition on $(X',L',G)$ ensures that there is an ample class $\alpha$ on $X$ such that for {\it any} torus equivariant reflexive sheaf $\cE$ that is stable on $(X,\alpha)$, the sheaf $(\iota_*(\pi^*\cE))^{\vee \vee}$ is stable on $(X',L')$ \cite{CT22}. Toric locally trivial fibrations can be seen as GIT quotients. However, the results presented in this paper are different in nature from the ones in \cite{CT22}. In the present situation, we fix a {\it single} sheaf and study the stability of its pullback for adiabatic polarisations, whereas in \cite{CT22}, the set of {\it all} pulled back reflexive equivariant sheaves is considered at once, for non necessarily adiabatic polarisations.
\end{rem}

\begin{rem}
In fact, we will see in Section \ref{sec:families} that our results hold for specific flat
families of equivariant reflexive sheaves. More precisely, if $(\cE_s)_{s\in S}$ is a family
of stable equivariant reflexive sheaves over $(X,L)$ with either
\begin{enumerate}[label=(\Roman*)]
\item $\cE$ is locally free on $X \times S$, or
\item the characteristic function $(\chi(\cE_s))_{s \in S}$ is constant,
\end{enumerate}
then the $\ep_0$ in Theorem \ref{theo:stablecaseintro} can be taken uniformly for
$(\cE_s)_{s\in S}$. Similarly, if we assume all $\cE_s$ to be sufficiently smooth, the
$\ep_0$ in Theorem \ref{theo:semistablecaseintro} can be taken uniformly in $s\in S$ provided
one of conditions (I) or (II) above is satisfied. As a corollary, we will see that the
reflexive pullback induces injective maps between the relevant moduli spaces of stable
equivariant sheaves. As those moduli spaces arise as fixed point loci under the torus action
on moduli spaces of reflexive sheaves on toric varieties \cite{Koo11}, we hope to extract
more information from those injective maps, at least for some simple fibrations.
\end{rem}

We then specify our results to various types of toric fibrations.
The first one that we address in Section \ref{sec:loctrivialfibration} is when $X'=X$ and
$\pi$ is the identity. The only modification comes then from the change in polarisation
from $L$ to $L + \ep L'$. As noticed in \cite{SekTip22}, in that case, our result already
provides interesting information on the behaviour of a semistable reflexive sheaf when the
polarisation varies. On a global level, moduli spaces of stable sheaves are subject to modifications related to wall-crossing phenomena in the ample cone (see \cite[Chapter 4, Section C]{HuLe} and reference therein for results on variations of moduli spaces of stable bundles on surfaces). Restricting to a single semistable reflexive sheaf $\cE$, Theorem \ref{theo:semistablecaseintro} gives a simple and effective criterion on perturbations of the polarisation that send $\cE$ to the stable locus.
As an illustration, we describe in Section \ref{sec:loctrivialfibration} stable perturbations of the tangent sheaf of a normal Del Pezzo surface, which is strictly semistable with respect to the anticanonical polarisation.

Another case of interest is when $\pi: X'\to X$ is an equivariant blow-up along a torus invariant subvariety $Z\subseteq X$ and $L_\ep = \pi^{\ast}L- \ep E$ where $E$ is the exceptional divisor of $\pi$. Assuming $X$ to be smooth, we push further the study of pulling back semistable sheaves under that setting in Section \ref{sec:blowups}. In  particular, if $S$ is a set of invariant points under the torus action of $X$, we obtain (see Section \ref{sec:blowup-along-point}):

\begin{theorem}\label{theo:blowuppoint}
Let $(X,L)$ be a smooth polarised toric variety and $S$ a set of invariant points under the
torus action of $X$. Let $\pi: X' \to X$ be the blow-up along $S$ and let
$L_{\ep}=\pi^*L-\ep E$ for $E$ the exceptional divisor of $\pi$. Let $\cE$ be a
$T$-equivariant reflexive sheaf that is strictly semistable on $(X,L)$. Then there is
$\ep_0>0$ such that for all $\ep\in (0,\ep_0)\cap \Q$, the reflexive pullback
$\cE':=(\pi^*\cE)^{\vee \vee}$ on $(X', L_\ep)$ is 
 \begin{itemize}
  \item[(i)] strictly semistable iff for any  subsheaf $\cF\in\fE$, $(\pi^*\cF)^{\vee \vee}$ is saturated in $\cE'$,
  \item[(ii)] unstable otherwise.
 \end{itemize}
\end{theorem}

\begin{cor}
 \label{cor:sufficientlysmoothblowuppoint}
 With the notations of Theorem \ref{theo:blowuppoint}, if $\cE$ is sufficiently smooth, then $\cE'$ satisfies $(i)$ and thus is strictly semistable on $(X', L_\ep)$ for $\ep \ll 1$.
\end{cor}

\begin{rem}
Corollary \ref{cor:sufficientlysmoothblowuppoint}, together with Theorem \ref{theo:stablecaseintro}, show that blowing-up points strictly preserves (semi)stability of a sufficiently smooth vector bundle for adiabatic polarisations.
However, in general, the reflexive pullback of a saturated subsheaf might not be saturated, see Example \ref{ex:nonsaturated}. Hence, pulling back along a single point blow-up  might ``destabilize'' a semistable reflexive sheaf.
\end{rem}
Hence, for adiabatic polarisations, blowing-up a point will never push a strictly semistable toric sheaf to the stable locus. This is no longer true if we blow-up higher dimensional varieties.
In Section \ref{sec:blowup-along-cure}, we prove :

\begin{theorem}\label{theo:blowupcurve}
Let $(X,L)$ be a smooth polarised toric variety. Let $\pi: X' \to X$ be the blow-up along
a $T$-invariant irreducible curve $Z \subseteq X$ and let $L_{\ep}=\pi^*L-\ep E$ for $E$ the
exceptional divisor of $\pi$. Let $\cE$ be a $T$-equivariant reflexive sheaf that is
strictly semistable on $(X,L)$. Then there is $\ep_0>0$ such that for all
$\ep\in (0,\ep_0)\cap \Q$, the pullback $\cE':=(\pi^*\cE)^{\vee \vee}$ on $(X', L_\ep)$ is
\begin{itemize}
 \item[(i)] stable iff for all $\cF\in \fE$, $(\pi^\ast \cF)^{\vee \vee}$ is saturated in
 $\cE'$ and $$\dfrac{c_1(\cE) \cdot Z}{\rk \cE} < \dfrac{c_1(\cF) \cdot Z}{\rk \cF};$$
 \item[(ii)] semistable iff for all $\cF\in\fE$, $(\pi^\ast \cF)^{\vee \vee}$ is
 saturated in $\cE'$ and $$\dfrac{c_1(\cE) \cdot Z}{\rk \cE} \leq \dfrac{c_1(\cF) \cdot Z}{\rk \cF};$$
 \item[(iii)] unstable otherwise.
 \end{itemize}
\end{theorem}

\begin{rem}
In Theorem \ref{theo:blowupcurve}, if $\cE$ is sufficiently smooth, then for all $\cF\in\fE$,
$(\pi^\ast \cF)^{\vee \vee}$ is saturated in $\cE'$ (cf. Lemma
\ref{lem:pullback-saturated-sheaf}). In that case, to study stability of $\cE'$ on
$(X', L_\ep)$ for $\ep \ll 1$, it is enough to compare the intersection numbers
$\frac{c_1(\cE) \cdot Z}{\rk \cE}$ and $\frac{c_1(\cF) \cdot Z}{\rk \cF}$ for $\cF$ in the
finite set $\fE$. We provide in Section \ref{sec:example-picard-rank2} an explicit
semistable example, namely the tangent sheaf of a Picard rank $2$ toric variety, that
becomes stable when pulled back to the blow-up along a curve. 
\end{rem}

Theorem \ref{theo:stablecaseintro} and Theorem \ref{theo:blowuppoint} recover and generalize
some of the results in \cite{Buch} and \cite{DerSek} on pullbacks of stable bundles along
blow-ups of points. While we restrict ourselves to toric varieties, our results cover the
cases of stable reflexive sheaves and semistable sufficiently smooth vector bundles. In
comparison, in \cite{Buch} the base manifold is a surface, while in \cite{DerSek}, the method
is via a gluing construction for Hermite-Einstein metrics, providing more precise information
on the behaviour of the metrics when $\ep\to 0$, but with a restriction to stable bundles.
The closer result in \cite[Proposition 5.1]{Graf-MR-for-linear-system} is more general than
our Corollary \ref{cor:sufficientlysmoothblowuppoint} as it deals with pullbacks of
semistable torsion-free sheaves over normal projective varieties, but Theorem
\ref{theo:blowuppoint} seems to provide more information when $\cE$ is not sufficiently
smooth.  On the other hand, Theorem \ref{theo:blowupcurve} seems to be, to the knowledge of
the authors, the first result in the direction of pushing a semistable bundle to the stable
locus by blowing-up higher dimensional sub-varieties.

The results in Theorem \ref{theo:blowupcurve} rely on a more general formula for slopes of
pullback sheaves under blow-ups. In general, if $Z\subseteq X$ is an $\ell$-dimensional
smooth subvariety of a smooth projective variety $X$ with $1 \leq \ell \leq \dim(X)-2$,
and $\cE$ is a reflexive sheaf on $X$, then, setting $\pi : X' \to X$ the blow-up along $Z$, we have:
\begin{equation}
\label{eq:formulaintro}
\mu_{L_\ep}((\pi^*\cE)^{\vee\vee})= \mu_{L}(\cE) - \dbinom{n-1}{\ell-1}
\mu_{L_{\vert Z}}(\cE_{|Z}) \ep^{n-\ell} + O(\ep^{n-\ell+1}) .
\end{equation}
This formula is quite striking as from Mehta-Ramanathan's restriction theorem, if $Z$ is
generic and an intersection of divisors coming from linear systems $H^0(X,L^{k_i})$ with
large $k_i$'s, then $\cE_{\vert Z}$ will be semistable provided $\cE$ is. Hence, in that
situation, formula (\ref{eq:formulaintro}) shows that subsheaves $\cF\subseteq\cE$ with
$\mu_L(\cF)= \mu_L(\cE)$ tend to destabilise $(\pi^*\cE)^{\vee\vee}$. Setting ourselves in
a typically non-generic situation, we can avoid this no go result. We obtain:
\begin{theorem}\label{theo:blowuphigher}
Let $(X,L)$ be a smooth polarised toric variety. Let $\pi: X' \to X$ be the blow-up along a
$T$-invariant irreducible subvariety $Z \subseteq X$ of codimension at least $2$ and let
$L_{\ep}=\pi^*L-\ep E$ for $E$ the exceptional divisor of $\pi$. Let $\cE$ be a
$T$-equivariant reflexive sheaf that is strictly semistable on $(X,L)$. Assume that for all
$\cF\in \fE$, $(\pi^\ast \cF)^{\vee \vee}$ is saturated in $\cE':=(\pi^*\cE)^{\vee \vee}$
and that
$$
 \mu_{L_{\vert Z}}(\cE_{|Z})< \mu_{L_{\vert Z}}(\cF_{|Z}).
$$
Then there is $\ep_0>0$ such that for all $\ep\in (0,\ep_0) \cap \Q$, the pullback $\cE'$
is stable on $(X', L_\ep)$.
\end{theorem}

Finally, Theorem \ref{theo:stablecaseintro} has another consequence on resolution of
singularities. An application of Hironaka's resolution of indeterminacy locus shows that
for a given equivariant reflexive sheaf $\cE$ on $X$, there is a finite sequence of
blow-ups along smooth irreducible torus invariant centers $\pi_i : X_i \to X_{i-1}$ for
$1 \leq i \leq p$ with $X_0=X$ such that, if we set $\cE_i=(\pi_i^*\cE_{i-1})^{\vee \vee}$,
the sheaf $\cE':=\cE_p$ is locally free on $X':=X_p$.
Each map $\pi_i$ is a toric fibration, and thus we can iterate Theorem
\ref{theo:stablecaseintro}. Starting with a stable sheaf $\cE$, we then obtain a stable
locally free sheaf $\cE'$ on $(X', L')$ that is isomorphic to $\cE$ away from the
exceptional locus.
%
In \cite{BS} and \cite{Sib}, a similar result is obtained, without the toric hypothesis,
but with differential geometric methods, and for a different polarisation on $X'$. The
polarisation in \cite{BS,Sib} is of the form $L+\ep H$, where $H$ is an ample divisor on
$X'$. In contrast, the polarisation $L'$, at the level of K\"ahler forms, only affects the
geometry of $X$ on a small neighborhood of the exceptional divisor. We believe that this
can be useful regarding the resolution of admissible Hermite-Einstein metrics as introduced
by Bando and Siu in \cite{BS}. We will come back to these explicit resolutions in the
forthcoming \cite{NapTip}.

\begin{rem}
We should note that while we restrict ourselves to toric varieties and equivariant sheaves,
we believe that all results in this paper should be true without the torus equivariant
assumption, on normal varieties. Nevertheless, working with equivariant structures provides
several crucial simplifications in the arguments, and enables to produce explicit examples
that might be difficult to find in general. For the sake of generality, we aim to relax our
equivariant hypothesis in future work.
\end{rem}

\subsection*{Organisation and conventions} All varieties considered in this paper are defined over the complex numbers and assumed to be normal. In Section \ref{sec:background} we recall the necessary background on toric varieties, their morphisms and their equivariant sheaves. We also recall the basics of slope stability. In Section \ref{sec:mainresult}, we prove Theorem \ref{theo:stablecaseintro}, Proposition \ref{prop:unstablecaseintro} and Theorem \ref{theo:semistablecaseintro}. We then give the first applications in the case of locally trivial fibrations. Section \ref{sec:blowups} is a more in depth study of the blow-up case, in which proofs of Theorems \ref{theo:blowuppoint},  \ref{theo:blowupcurve} and \ref{theo:blowuphigher}, as well as Corollary \ref{cor:sufficientlysmoothblowuppoint}, together with applications, are given.

\subsection*{Acknowledgments}  
The authors would like to thank Lars Martin Sektnan for stimulating discussions on this
problem, Ruadha\'i Dervan for several enlightening suggestions and Michel Brion for his
careful reading of the manuscript and his advice.
We also thank the anonymous referee for their advices that improved the exposition and the content of this paper.
The authors are partially supported by
the grants MARGE ANR-21-CE40-0011 and BRIDGES ANR--FAPESP ANR-21-CE40-0017.

\section{Background}
\label{sec:background}
In this first section we gather the necessary background about toric varieties \cite{CLS} and equivariant reflexive sheaves \cite{Kly90,Per03}.

\subsection{Toric varieties and divisors}
\label{sec:toric-varieties}
Let $N$ be a rank $n$ lattice and $M$ be its dual with pairing
$\langle \, \cdot \, , \, \cdot \, \rangle : M \times N \rightarrow \Z$. The lattice $N$ is the
lattice of one-parameter subgroups of $N \otimes_{\Z} \C^{\ast}$.
For $\K = \R ~ \text{or} ~ \C$, we define $N_{\K} = N \otimes_{\Z} \K$ and
$M_{\K} = M \otimes_{\Z} \K$.
A {\it fan} $\Sigma$ in $N_{\R}$ is a set of rational strongly convex polyhedral cones
in $N_{\R}$ such that:
\begin{itemize}
\item Each face of a cone in $\Sigma$ is also a cone in $\Sigma$;
\item The intersection of two cones in $\Sigma$ is a face of each.
\end{itemize}
We will denote $\tau \preceq \sigma$ the inclusion of a face $\tau$ in $\sigma \in \Sigma$.
A cone $\sigma$ in $N_{\R}$ is {\it smooth} if its minimal generators form part of a
$\Z$-basis of $N$. We say that $\sigma$ is {\it simplicial} if its minimal generators are
linearly independent over $\R$.
A fan $\Sigma$ is {\it smooth} (resp. {\it simplicial}) if every cone $\sigma$ in $\Sigma$
is smooth (resp. {\it simplicial}).
The {\it support} of $\Sigma$ is $|\Sigma| := \bigcup_{\sigma \in \Sigma}{\sigma}\,$.
We say that $\Sigma$ is {\it complete} if $|\Sigma| = N_{\R}$.

For $\sigma \in \Sigma$, let $U_{\sigma} = \Spec( \C[S_{\sigma}])$ where
$\C[S_{\sigma}]$ is the semi-group algebra of
$$
S_{\sigma} = \sigma^{\vee} \cap M = \{ m \in M \, : \, \langle m, \, u \rangle \geq 0
~ \text{for all}~ u \in \sigma \} ~.
$$
If $\sigma, \, \sigma' \in \Sigma$, we have $U_{\sigma} \cap U_{\sigma'} =
U_{\sigma \cap \sigma'}$.
We denote by $X_{\Sigma}$ the toric variety associated to a fan $\Sigma$;
$X_{\Sigma}$ is obtained by gluing the affine charts $(U_{\sigma})_{\sigma \in \Sigma}$.
The variety $X_{\Sigma}$ is normal and its torus is $T = N \otimes_{\Z} \C^{\ast}$.
By \cite[Theorem 3.1.19]{CLS}, the toric variety $X_\Sigma$ is {\it smooth}
(resp. $\Q$-{\it factorial}) if and only if the fan $\Sigma$ is smooth (resp. simplicial).

Let $X$ be the toric variety associated to a fan $\Sigma$ in $N_{\R}$.
For any $\sigma \in \Sigma$, there is a point $\gamma_\sigma \in U_{\sigma}$ called the
{\it distinguished point} of $\sigma$ such that the torus orbit $O(\sigma)$ corresponding to $\sigma$ is given by $O(\sigma) = T \cdot \gamma_\sigma$. We will use the following result:

\begin{theorem}[Orbit-Cone Correspondence, {\cite[Theorem 3.2.6]{CLS}}]\label{theo:orbit-cone}
$\,$
\begin{enumerate}
\item
There is a bijective correspondence
$$
\begin{array}{rcl}
\{ \text{Cones} ~ \sigma ~ \text{in} ~ \Sigma \} & \eqv & \{ T-\text{orbits in} ~ X \}
\\ \sigma & \eqv & O(\sigma)
\end{array}
$$
with
$\dim O(\sigma) = \dim N_{\R} - \dim \sigma$.
\item
The affine open subset $U_{\sigma}$ is the union of orbits
$$~ \disp{
U_{\sigma} = \bigcup_{\tau \preceq \sigma}{O(\tau)} }~.$$
\item We have
$\tau \preceq \sigma$ if and only if $O(\sigma) \subseteq\overline{O(\tau)}$. 
\item Finally,
$$~\disp{
\overline{O(\tau)} = \bigcup_{\tau \preceq \sigma}{O(\sigma)}
}~$$
where $\overline{O(\tau)}$ denotes the closure in both the classical and Zariski
topologies.
\end{enumerate}
\end{theorem}

\begin{notation}
We will use these notations.
\begin{itemize}
\item
For $m \in M$, we denote by $\chi^m: T \rightarrow \C^\ast$ the corresponding character.
\item
$\Sigma(k)$ is the set of $k$-dimensional cones of $\Sigma$ and for any $\sigma \in \Sigma$,
$\sigma(1) = \Sigma(1) \cap \{ \tau \in \Sigma : \tau \preceq \sigma \}$.
\item
For $\rho \in \Sigma(1)$, we denote by $u_{\rho} \in N$ the minimal generator of $\rho$.
\item
For $\sigma \in \Sigma$, we set $V(\sigma) = \overline{O(\sigma)}$.
If $\rho \in \Sigma(1)$, we denote $V(\rho)$ by $D_{\rho}$.
\item
For any $\sigma \in \Sigma$, we set $N_{\sigma} = \Span(\sigma) \cap N$ and
$M(\sigma) = M \cap \sigma^{\perp}$.
\end{itemize}
\end{notation}

Divisors of the form $\sum_{\rho \in \Sigma(1)} a_{\rho} D_{\rho}$ are precisely the invariant
divisors under the torus action on $X_{\Sigma}$. Thus,
$$
\WDiv_T(X_\Sigma) := \bigoplus_{\rho \in \Sigma(1)} \Z D_\rho
$$
is the group of invariant Weil divisors on $X_{\Sigma}$. We denote by $\Div_T(X_{\Sigma})$
the set of invariant Cartier divisors on $X_{\Sigma}$.
A {\it support function} of $\Sigma$ is a function
$\varphi: |\Sigma| \rightarrow \R$ that is linear on each cone of $\Sigma$. Support
functions can be used to characterize Cartier divisors:

\begin{prop}[{\cite[Theorem 4.2.12]{CLS}}]
\label{prop:support-function-divisor}
Let $D = \sum_{\rho \in \Sigma(1)}{a_{\rho}D_{\rho}}$ be a Cartier divisor of
$X_{\Sigma}$. The support function $\varphi_{D} : |\Sigma| \rightarrow \R$ associated to the
divisor $D$ is defined by $\varphi_{D}(u_{\rho}) = - a_{\rho}$ for any
$\rho \in \Sigma(1)$.
\end{prop}

We now provide formulas that will be used to compute various intersections of toric
divisors. We assume that $X$ is an $n$-dimensional toric variety given by a complete and
simplicial fan $\Sigma$.
An element $u \in N$ is {\it primitive}\index{primitive element} if
$\frac{1}{k} u \notin N$ for all $k >1$.
Let $\{u_1, \ldots, u_k\}$ be a set of primitive elements of $N$ such that
$\sigma = \Cone(u_1, \ldots, u_k)$ is simplicial. We define $\mult(\sigma)$ as the index
of the sublattice $\Z u_1 + \ldots + \Z u_k$ in $N_{\sigma} = \Span(\sigma) \cap N$.

As $\Sigma$ is simplicial, according to \cite[Section 5.1]{Ful93}, one has intersections
of cycles or cycle classes only with rational coefficients. The Chow group
$$
A^*(X)_\Q = \bigoplus_{p=0}^{n} A^{p}(X) \otimes \Q =
\bigoplus_{p=0}^{n} A_{n-p}(X) \otimes \Q
$$
has the structure of graded $\Q$-algebra and,

\begin{prop}\label{prop:intersection-toric-simplicial}
Let $\tau, \tau', \sigma \in \Sigma$ such that $\tau$ and $\tau'$ span $\sigma$, with
$\dim(\sigma) = \dim(\tau) + \dim(\tau')$, then
$$
[V(\tau)] \cdot [V(\tau')] = \dfrac{\mult(\tau) \cdot \mult(\tau')}{\mult(\sigma)}
[V(\sigma)] .
$$
\end{prop}
This proposition is a consequence of the following Lemmas.

\begin{lem}[{\cite[Lemma 12.5.1]{CLS}}]
The Chow group $A_k(X)$ is generated by the classes of the orbit closures $V(\sigma)$ of
the cones $\sigma \in \Sigma$ of dimension $n-k$. 
\end{lem}

\begin{lem}[{\cite[Lemma 12.5.2]{CLS}}]\label{lem:intersection-toric-simplicial}
Assume that $\Sigma$ is complete and simplicial.
If $\rho_1, \ldots, \rho_d \in \Sigma(1)$ are distinct, then in $A^{\bullet}(X)_\Q$,
we have
$$
[D_{\rho_1}] \cdot [D_{\rho_2}] \cdots [D_{\rho_d}] = \left\lbrace
\begin{array}{ll}
\dfrac{1}{\mult(\sigma)} [ V(\sigma)] & \text{if} ~
\sigma = \rho_1 + \ldots + \rho_d \in \Sigma \\ 0 & \text{otherwise.}
\end{array}
\right.$$
\end{lem}

%
\begin{proof}[Proof of Proposition \ref{prop:intersection-toric-simplicial}]
Let $\rho_1, \ldots, \rho_q \in \Sigma(1)$ distinct such that
$\tau = \rho_1 + \ldots + \rho_p$ and $\tau'= \rho_{p+1} + \ldots \rho_q$ with $p<q$.
By Lemma \ref{lem:intersection-toric-simplicial}, we get
\begin{align*}
\dfrac{1}{\mult(\sigma)} [V(\sigma)] & = ([D_{\rho_1}] \cdots [D_{\rho_p}]) \cdot
([D_{\rho_{p+1}}] \cdots [D_{\rho_q}])
\\ & = \dfrac{1}{\mult(\tau) \cdot \mult(\tau')} [V(\tau)] \cdot [V(\tau')].
\end{align*}
\end{proof}

For $m \in M$, the character $\chi^m$ is a rational function on $X$. By
\cite[Proposition 4.1.2]{CLS} the divisor of $\chi^m$ is given by
\begin{equation}\label{eq:divisor-of-character}
\ddiv(\chi^m) =\sum_{\rho \in \Sigma(1)}{ \langle m, \, u_{\rho} \rangle} D_{\rho}
\end{equation}
and $\ddiv(\chi^m) = 0$ in $A_{n-1}(X)$.
In the case where $\rho \in \Sigma(1)$ is a ray of $\sigma \in \Sigma$, there is
$m \in M$ such that $V(\sigma)$ is not contained in the support of
$D_{\rho} + \ddiv(\chi^m)$. We then set
\begin{equation}\label{eq:sel-intersection}
[D_{\rho}] \cdot [V(\sigma)] = [D_{\rho} + \ddiv(\chi^m)] \cdot [V(\sigma)] .
\end{equation}

\subsection{Toric morphisms}\label{sec:toric-morphism}
For this part we refer to \cite[Section 3.3]{CLS}.
Let $N_1$, $N_2$ be two lattices with $\Sigma_1$ a fan in $(N_1)_{\R}$ and $\Sigma_2$
a fan in $(N_2)_{\R}$. We denote by $X_1$ (resp. $X_2$) the toric variety associated
to the fan $\Sigma_1$ (resp. $\Sigma_2$).

A morphism $\pi : X_1 \rightarrow X_2$ is {\it toric} if $\pi$ maps the torus $T_1$ of $X_1$
into the torus $T_2$ of $X_2$ and $\pi_{| T_1}: T_1 \rightarrow T_2$ is a group homomorphism.
We say that a  $\Z$-linear map $\phi: N_1 \rightarrow N_2$ is {\it compatible} with the fan
$\Sigma_1$ and $\Sigma_2$ if for every cone $\sigma_1 \in \Sigma_1$, there is a cone
$\sigma_2 \in \Sigma_2$ such that $\phi_{\R}(\sigma_1) \subseteq\sigma_2$ where $\phi_{\R}$
is the $\R$-linear extension of $\phi$. According to  \cite[Theorem 3.3.4]{CLS}, any toric
morphism $\pi: X_1 \rightarrow X_2$ comes from a $\Z$-linear map $\phi: N_1 \rightarrow N_2$
that is compatible with $\Sigma_1$ and $\Sigma_2$.

We will be interested in {\it toric fibrations}. A proper toric morphism $\pi: X_1 \rightarrow X_2$ is a {\it fibration} if $\pi_\ast( \cO_{X_1}) = \cO_{X_2}$. According to \cite[Theorem 3.4.11]{CLS} and \cite[Proposition 2.1]{Cataldo-toricmaps}, if $X_1$ and $X_2$ are complete, $\pi$ is a toric fibration if and only if the associated map $\phi: N_1 \rightarrow N_2$ is surjective.
We now give specific examples of toric fibrations.

\subsubsection{Locally trivial fibrations}\label{sec:fibration}
Let $N$, $N'$ be two lattices and $\phi : N' \rightarrow N$ be a surjective
$\Z$-linear map. Let $\Sigma'$ be a fan in $N_{\R}'$ and $\Sigma$ a fan in $N_{\R}$
compatible with $\phi$. We set $N_0 = \Ker( \phi)$ and
$\Sigma_0 = \{ \sigma \in \Sigma' : \sigma \subseteq (N_0)_{\R} \}$.
We have an exact sequence
\begin{equation}\label{eq:split-lattice}
0 \imp N_0 \imp N' \imp N \imp 0
\end{equation}
We say that $\Sigma'$ is {\it weakly split by} $\Sigma$ and $\Sigma_0$ if there exists a subfan $\widehat{\Sigma}$ of $\Sigma'$ such that:
\begin{enumerate}
\item
$\phi_{\R}$ maps each cones $\widehat{\sigma} \in \widehat{\Sigma}$ bijectively
to a cone $\sigma \in \Sigma$. Furthermore, the map $\widehat{\sigma} \mapsto \sigma$ define a
bijection $\widehat{\Sigma} \rightarrow \Sigma$.
\item
Given $\widehat{\sigma} \in \widehat{\Sigma}$ and $\sigma_0 \in \Sigma_0$, the sum
$\widehat{\sigma} + \sigma_0$ lies in $\Sigma'$, and every cone of $\Sigma'$ arises this
way.
\end{enumerate}
Moreover, if $\phi(\widehat{\sigma} \cap N')= \sigma \cap N$ for any
$\widehat{\sigma} \in \widehat{\Sigma}$ with $\phi_{\R}(\widehat{\sigma}) = \sigma$, we say
that $\Sigma'$ is {\it split by} $\Sigma$ and $\Sigma_0$.

\begin{theorem}[{\cite[Theorem 3.3.19]{CLS}}]\label{theo:fiber-bundle}
If $\Sigma'$ is split by $\Sigma$ and $\Sigma_0$, then $X_{\Sigma'}$ is a locally trivial
fiber bundle over $X_{\Sigma}$ with fiber $X_{\Sigma_0, N_0}$ where $X_{\Sigma_0, N_0}$ is
the toric variety associated to the fan $\Sigma_0$ in $(N_0)_{\R}$.
\end{theorem}

In the case where $\Sigma'$ is weakly split by $\Sigma$ and $\Sigma_0$, for any
$\sigma \in \Sigma$ there is a sublattice $N'' \subseteq N$ of finite index such that
$\Sigma'(\sigma) = \{\sigma' \in \Sigma' : \phi_{\R}(\sigma') \subseteq \sigma \}$ is split
by $\{ \tau \in \Sigma : \tau \preceq \sigma \}$ and $\Sigma'(\sigma) \cap \Sigma_0$
when we use the lattice $\phi^{-1}(N'')$ and $N''$. Let $U_{\sigma, N''}$ be the toric
variety associated to the cone $\sigma$ in $(N'')_\R$. There is a commutative diagram
$$
\begin{tikzcd}
X_{\Sigma_0, N_0} \times U_{\sigma, N''} \arrow[rr] \arrow[d] & & X_{\Sigma'} \arrow[d]
\\ U_{\sigma, N''} \arrow[rr] & & ~U_{\sigma}
\end{tikzcd}
$$
such that $X_{\Sigma_0, N_0} \times U_{\sigma, N''}$ is the fiber product
$X_{\Sigma'} \times_{U_{\sigma}} U_{\sigma, N''}$.

\begin{cor}\label{cor:fiber-bundle}
If $\Sigma'$ is weakly split by $\Sigma$ and $\Sigma_0$, then the fibers of
$\pi : X_{\Sigma'} \rightarrow X_{\Sigma}$ are isomorphic to $X_{\Sigma_0, N_0}$.
\end{cor}

\subsubsection{Blowups}\label{sec:blowup-of-variety}
We recall here the characterization of the blowup of a toric variety along an invariant subvariety.
Let $\Sigma$ be a fan in $N_{\R}$ and assume $\tau \in \Sigma$ with $\dim \tau \geq 2$ has
the property that all cones of $\Sigma$ containing $\tau$ are smooth. Let
$u_{\tau} = \sum_{\rho \in \tau(1)}{u_{\rho}}$ and for each cones $\sigma \in \Sigma$
containing $\tau$, set
$$
\Sigma_{\sigma}^{\ast}(\tau) = \{ \Cone(A) : A \subseteq\{u_{\tau} \} \cup \sigma(1)
~ \text{and} ~ \tau(1) \nsubseteq A \} ~.
$$
The {\it star subdivision} of $\Sigma$ relative to $\tau$ is the fan
$$
\Sigma^{\ast}(\tau) = \{ \sigma \in \Sigma : \tau \nsubseteq \sigma \} \cup
\bigcup_{\tau \subseteq \sigma} \Sigma_{\sigma}^{\ast}(\tau) ~.
$$
A fan $\Sigma'$ {\it refines} $\Sigma$ if every cone of $\Sigma'$ is contained in a cone of $\Sigma$ and $|\Sigma'| = |\Sigma|$. When $\Sigma'$ refines $\Sigma$, the identity mapping $\phi = \Id_N$ is compatible with $\Sigma'$ and $\Sigma$.
So there is a toric morphism $\pi : X_{\Sigma^{\ast}(\tau)} \rightarrow X_{\Sigma}$.
Under $\pi$, $X_{\Sigma^{\ast}(\tau)}$ is the blowup of $X_{\Sigma}$ along $V(\tau)$ and the exceptional divisor $D_0$ of $\pi$ is the divisor corresponding to the ray $\Cone(u_{\tau})$ of $\Sigma^{\ast}(\tau)$.

\subsection{Reflexive sheaves}
Let $X$ be a smooth toric variety associated to a fan $\Sigma$ in $N_\R$. Recall that a reflexive sheaf on $X$ is a coherent sheaf $\cE$ that is canonically isomorphic to its double dual $\cE^{\vee\vee}$.

Let $\theta : T \times X \rightarrow X$ be the action of $T$ on $X$,
$\mu : T \times T \rightarrow T$ the group multiplication, $p_2 : T \times X \rightarrow X$
the projection onto the second factor and
$p_{23} : T \times T \times X \rightarrow T \times X$ the projection onto the second and
the third factor. We call a sheaf $\cE$ on $X$ {\it equivariant} if it is equipped with
an isomorphism $\Phi : \theta^{\ast} \cE \rightarrow p_{2}^{\ast} \cE$ such that
\begin{equation}\label{eq:cocyle-equivariant-sheaf}
(\mu \times \Id_X)^{\ast} \Phi = p_{23}^{\ast} \Phi \circ
(\Id_T \times \theta)^{\ast} \Phi ~.
\end{equation}

Klyachko gave a description of torus equivariant reflexive sheaves over toric varieties in terms of combinatorial data \cite{Kly90}:
\begin{defn}
 \label{def:family of filtrations}
 A family of filtrations $\E$ is the data of a finite dimensional vector space $E$ and for each ray $\rho\in\Sigma(1)$, an increasing filtration $(E^\rho(i))_{i\in\Z}$ of $E$ such that $E^\rho(i)=\lbrace 0 \rbrace$ for $i\ll 0$ and $E^\rho(i)=E$ for some $i$. 
 \end{defn}
 \begin{rem}
 Note that we are using increasing filtrations here, as in \cite{Per03}, rather than decreasing as in \cite{Kly90}.
\end{rem}
To a family of filtrations $\E:=
\left( E, \, \{ E^{\rho}(j) \}_{\rho \in \Sigma(1), \, j \in \Z} \right)$,  we can assign an equivariant reflexive sheaf $\cE:=\fK(\E)$ defined by 
 \begin{equation}
  \label{eq:sheaf from family of filtrations}
  \Gamma(U_{\sigma}, \cE):=\bigoplus_{m\in M} \bigcap_{\rho\in\sigma(1)} E^\rho(\langle m,u_\rho\rangle)\otimes \chi^m
 \end{equation}
 for all positive dimensional cones $\sigma\in\Sigma$, while $\Gamma(U_{\lbrace 0\rbrace},\cE)=E\otimes \mathbb{C}[M]$. The morphisms between families of filtrations are linear maps preserving the filtrations. Then, by \cite{Kly90}, the functor $\fK$ induces an equivalence of categories between the families of filtrations and equivariant reflexive sheaves over $X$.

\begin{notation}\label{nota:decomposition-of-epuiv-sheaf}
For any $\sigma \in \Sigma$, we denote $\Gamma(U_\sigma,\cE)$ by $E^\sigma$ and we set
$$E^\sigma = \bigoplus_{m \in M}{E_{m}^\sigma \otimes \chi^m} ~.$$
For any $\rho \in \Sigma(1)$ and $m \in M$, we set
$E_{m}^{\rho} = E^{\rho}(\< m, u_\rho \>)$.
\end{notation}

\begin{example}[Tangent sheaf]\label{examp:filtration-tangent}
The family of filtrations of the tangent sheaf $\cT_X$ of $X$ is given by
$$
E^{\rho}(j) = \left\lbrace
\begin{array}{ll}
0 & \text{if}~ j< -1 \\
\Span(u_{\rho}) & \text{if}~ j = -1 \\
N \otimes_{\Z} \C & \text{if} ~ j> -1
\end{array}
\right. ~.
$$
\end{example}

\subsection{Some stability notions}
For this part, we refer to \cite{Tak72} and \cite[Section 4.1]{CT22}.
Let $\cE$ be a torsion-free coherent sheaf on $X$. The {\it degree} of $\cE$ with
respect to an ample class $L \in \Amp(X)$ is the real number obtained by intersection:
$$
\deg_{L}(\cE)= c_1(\cE) \cdot L^{n-1}
$$
and its {\it slope} with respect to $L$ is given by
$$
\mu_{L}(\cE) = \dfrac{ \deg_{L}(\cE) }{\rk(\cE) }.
$$

\begin{defn}
A torsion-free coherent sheaf $\cE$ is said to be {\it slope semistable} (or
{\it semistable} for short) with respect to $L \in \Amp(X)$ if for any proper coherent
subsheaf of lower rank $\cF$ of $\cE$, one has
$$
\mu_{L}(\cF) \leq \mu_{L}(\cE) .
$$
When strict inequality always holds, we say that $\mathcal{E}$ is {\it stable}.
Finally, $\cE$ is said to be {\it polystable} if it is the direct sum of
stable subsheaves of the same slope.
\end{defn}

If $\cE$ is an equivariant reflexive sheaf on a normal toric variety $X$, according to
\cite[Proposition 4.13]{Koo11} and \cite[Proposition 2.3]{HNS19}, it is enough to
test slope inequalities for equivariant and reflexive saturated subsheaves.

\begin{prop}\label{prop:slope-with-saturated-sheaf}
Let $\cE$ be an equivariant reflexive sheaf on $X$. Then $\cE$ is semistable
(resp. stable) with respect to $L$ if and only if for all saturated equivariant
reflexive subsheaves $\cF$ of $\cE$, $\mu_L(\cF)\leq \mu_L(\cE)$ (resp. $\mu_L(\cF)< \mu_L(\cE)$).
\end{prop}

Let $\cE$ be an equivariant reflexive sheaf on a normal toric variety $X$ given by the
family of filtrations
$\left( E, \, \{ E^{\rho}(j) \}_{\rho \in \Sigma(1), \, j \in \Z} \right)$. From the previous proposition, it is crucial for us to understand the description of equivariant reflexive and saturated subsheaves of $\cE$ in terms of families of filtrations. This is the content of the following lemma.

\begin{lem}[{\cite{DDK20b}}]
\label{lem:saturated-subsheaf}
Let $F$ be a vector subspace of $E$ and $\cF$ an equivariant reflexive subsheaf of
$\cE$ given by the family of filtrations $(F, \, \{ F^\rho(i) \})$ with
$F^{\rho}(i) \subseteq E^{\rho}(i)$.
Then $\cF$ is saturated in $\cE$ if and only if for all $\rho \in \Sigma(1), i \in \Z$,
$$F^\rho(i) = E^\rho(i) \cap F.$$
\end{lem}

\begin{proof}
As $\cF$ is an equivariant sheaf, the quotient sheaf $\cE / \cF$ is equivariant.
For any $\rho \in \Sigma(1)$, one has
$\Gamma(U_\rho, \cE / \cF) = \Gamma(U_\rho, \cE)/ \Gamma(U_\rho, \cF)$
and
$$
\Gamma(U_\rho, \cE/ \cF)_{m} =
E^\rho( \<m, u_\rho \>)/ F^\rho(\<m, u_\rho \>) \otimes \chi^{m}
$$
for any $m \in M$.

We assume that there is $(\rho, \, i) \in \Sigma(1) \times \Z$ such that
$F^\rho(i) \neq E^{\rho}(i) \cap F$. Let
$$
i_0 = \min \{ i \in \Z : F^\rho(i) \neq E^\rho(i) \cap F \}
$$
and $v_\rho \in M$ such that $M_{\R} = \R v_\rho \oplus \Span(u_{\rho})^\perp$ with
$\<v_\rho, u_\rho \> = 1$. There is $e \in F \cap (E^\rho(i_0) \setminus F^\rho(i_0))$
such that $0 \neq \overline{e} \in E^\rho(i_0) / F^\rho(i_0)$ where $\overline{e}$
denotes the image of $e$ in $E^{\rho}(i_0)/ F^{\rho}(i_0)$.
Then, as $e \in F$ and as $F^\rho(i)=F$ for $i$ large enough, there is $m \in S_\rho$
with $\<m + i_0 v_\rho, u_\rho \> = i$, such that
$$
e \otimes \chi^{m + i_0 v_\rho} \in F^\rho(\<m + i_0 \, v_\rho, u_\rho \>) \otimes
\chi^{m + i_0 v_\rho} = F \otimes \chi^{m + i_0 v_\rho}.
$$
Thus $\overline{e} \otimes \chi^{m + i_0 v_\rho} = 0$ in $\Gamma(U_{\rho}, \cE / \cF)$.
Hence, $\cE / \cF$ has nonzero torsion.
Therefore, if $\cE / \cF$ is torsion-free, then $F^\rho(i) = E^\rho(i) \cap F$ for
any $(\rho, i) \in \Sigma(1) \times \Z$.

We now assume that for any $(\rho, i) \in \Sigma(1) \times \Z$,
$F^\rho(i) = E^\rho(i) \cap F$. We use Notation \ref{nota:decomposition-of-epuiv-sheaf}.
For any $\sigma \in \Sigma$, we set
$$
E^\sigma = \Gamma(U_\sigma, \cE) = \bigoplus_{m \in M} E_{m}^\sigma \otimes \chi^m
\quad \text{and} \quad
F^\sigma = \Gamma(U_\sigma, \cF) = \bigoplus_{m \in M} F_{m}^\sigma \otimes \chi^m .
$$
We will show that $E^\sigma / F^\sigma$ is a torsion-free $\C[S_\sigma]$-module.
Let $e \in E_{m}^{\sigma}$ and $\overline{e}$ be its image in
$E_{m}^\sigma / F_{m}^\sigma$ such that there is $m' \in S_\sigma$ with
$\overline{e} \otimes \chi^{m+m'} = 0$
in $\Gamma(U_\sigma, \cE)_{m+m'} / \Gamma(U_\sigma, \cF)_{m+m'}$.
We have
$$
e \otimes \chi^{m+m'} \in \Gamma(U_\sigma, \cF)_{m+m'}
$$
and then
$e \in F_{m+m'}^\sigma \subseteq F$. As $F_{m}^\sigma = E_{m}^\sigma \cap F$, we get
$e \in F_{m}^\sigma$. Hence, $E^\sigma / F^\sigma$ is a torsion-free
$\C[S_\sigma]$-module. Therefore, $\cE / \cF$ is torsion-free.
\end{proof}

\begin{notation}
\label{notation:subsheaf}
Let $F$ be a vector subspace of $E$. We denote by $\cE_F$ the saturated subsheaf of
$\cE$ defined by the family of filtrations $\left( F, \, \{ F^{\rho}(j) \} \right)$
where $F^{\rho}(j)= F \cap E^{\rho}(j).$
\end{notation}

By \cite[Corollary 3.18]{Koo11}, the first Chern class of an equivariant reflexive sheaf
$\cE$ with family of filtrations $(E, \{ E^{\rho}(j) \})$ is given by
\begin{equation}\label{eq:c1-sheaf}
c_1(\cE) = - \sum_{\rho \in \Sigma(1)} \iota_{\rho}(\cE) \, D_{\rho}
\end{equation}
where
$$
\iota_{\rho}(\cE) = \sum_{j \in \Z}{j \left( \dim(E^{\rho}(j)) - \dim(E^{\rho}(j-1)) \right)}.
$$
Therefore, for any $L \in \Amp(X)$,
\begin{equation}\label{eq:slope-sheaf}
\mu_{L}(\cE) = - \dfrac{1}{\rk(\cE)} \sum_{\rho \in \Sigma(1)} \iota_{\rho}(\cE)
\deg_{L}(D_{\rho}) .
\end{equation}
Thanks to Lemma \ref{lem:saturated-subsheaf}, if $\cE$ is an equivariant reflexive sheaf,
we have the following control on the number of values used in comparing slopes:

\begin{lem}\label{lem:slope-number-vector-space}
The set $\{ \mu_{L}(\cE_F) : F \subseteq E ~\text{with}~ 0< \dim F < \dim E \}$ is finite.
\end{lem}

\begin{proof}
For any $\rho \in \Sigma(1)$, there is $(j_{\rho}, J_\rho) \in \Z^2$ such that
$E^{\rho}(j) = \{0\}$ if $j < j_\rho$ and $E^{\rho}(j) = E$ if $j \geq J_{\rho}$.
For any vector subspace $F$ of $E$, we have
$$
\iota_{\rho}(\cE_F) = \sum_{j= j_\rho}^{J_{\rho}}{j \left( \dim(E^{\rho}(j)\cap F) -
\dim(E^{\rho}(j-1)\cap F) \right)} .
$$
As the set
$\{\dim(E^{\rho}(j)\cap F) - \dim(E^{\rho}(j-1)\cap F): F \subsetneq E, \, j \in \Z \}$
is finite, we deduce that $\{ \iota_{\rho}(\cE_F) : F \subsetneq E \}$ is finite.
We can conclude using Formula (\ref{eq:slope-sheaf}).
\end{proof}

\section{Proof of the main result}\label{sec:mainresult}

Let $N$ and $N'$ be two lattices having respectively $M$ and $M'$ for dual lattices.
Let $\Sigma$ be a complete fan in $N_{\R}$ and $\Sigma'$ a complete fan in $N'_{\R}$.
We denote by $X$ (resp. $X'$) the toric variety associated to the fan $\Sigma$
(resp. $\Sigma'$) and $T$ (resp. $T'$) its torus.
Let $\phi: N' \rightarrow N$ be a surjective $\Z$-linear map compatible with $\Sigma'$ and
$\Sigma$, and denote by $\pi: X' \rightarrow X$ the induced toric fibration. Let $\cE$ be an equivariant reflexive sheaf on $X$ and $\cE'= (\pi^\ast \cE)^{\vee \vee}$
its reflexive pullback on $X'$.
In this section, we will give the proofs of Theorem \ref{theo:stablecaseintro}, Theorem
\ref{theo:semistablecaseintro} and Proposition \ref{prop:unstablecaseintro}.

\subsection{Pulling back sheaves on a fibration}
We first describe part of the family of filtrations of the pulled back sheaves.
We will need the following lemmas.

\begin{lem}
Let $\rho \in \Sigma(1)$, there is $\rho' \in \Sigma'(1)$ such that $\phi_{\R}(\rho') = \rho$.
\end{lem}

\begin{proof}
Let $\rho \in \Sigma(1)$. We first note that there is $\sigma' \in \Sigma'$ such that
$\rho \subseteq\phi_{\R}(\sigma')$.

Let $\sigma' \in \Sigma'$ such that $\rho \subseteq\phi_{\R}(\sigma')$ and
$W = \phi_{\R}^{-1}(\Span(u_{\rho}))$. We set $\sigma'_0 = \sigma' \cap W$. As $\sigma'$ is a strongly convex polyhedral cone, we deduce that $\sigma'_0$ is also a strongly convex polyhedral cone and $\sigma'_0$ is a face of $\sigma'$.
As $\phi$ is compatible with $\Sigma'$ and $\Sigma$, there is $\sigma \in \Sigma$ such that
$\phi_{\R}(\sigma') \subseteq \sigma$. By strong convexity of $\sigma$, $\phi_{\R}(\sigma')$
doesn't contain $\Span(u_{\rho})$; hence, $\phi_{\R}(\sigma'_0) = \rho$.

If for any $\rho' \in \sigma'_0(1)$, $\phi_{\R}(\rho') \neq \rho$, then
$\phi_{\R}(\rho') = \{0\}$ by strong convexity of $\sigma$. By linearity of $\phi_{\R}$, we get $\phi_{\R}(\sigma'_0)= \{0\}$, which is a contradiction. Hence, there is
$\rho' \in \sigma'_0(1) \subseteq\Sigma'(1)$ such that $\phi_{\R}(\rho') = \rho$.
\end{proof}

\begin{lem}[{\cite[Lemma 3.3.21]{CLS}}]\label{lem:image-of-t-morphism}
Given $\sigma' \in \Sigma'$, let $\sigma$ be the minimal cone of $\Sigma$ containing $\phi_{\R}(\sigma')$. Then:
\begin{enumerate}
\item
$\pi \left( \gamma_{\sigma'} \right) = \gamma_{\sigma}$ where
$\gamma_{\sigma'} \in O(\sigma')$ and $\gamma_{\sigma} \in O(\sigma)$ are the
distinguished points.
\item
$\pi (O(\sigma')) = O(\sigma)$ and $\pi (V(\sigma')) = V(\sigma)$.
\end{enumerate}
\end{lem}
Note that in general, the equalities in point $(2)$ of Lemma \ref{lem:image-of-t-morphism} are only inclusions, but as $\pi$ is surjective, we actually get equalities. We can now obtain a partial description of pulled back sheaves in terms of
families of filtrations.

\begin{prop}\label{prop:pullback-sheaf-surjection}
Let $\cE$ be an equivariant reflexive sheaf on $X$ given by the family of filtrations
$\left( E, \, \{ E^{\rho}(j) \}_{\rho \in \Sigma(1), \, j \in \Z} \right)$. Let
$\left( \widetilde{E}, \, \{ \widetilde{E}^{\rho'}(j) \}_{\rho' \in \Sigma'(1), \,
j \in \Z} \right)$ be the family of filtrations of the equivariant sheaf
$(\pi^{\ast} \cE)^{\vee\vee}$. Then we have:
\begin{enumerate}
\item $\widetilde{E} = E \,$.
\item
If $\phi_{\R}(\rho') = \{0 \}$, then
$~
\widetilde{E}^{\rho'}(j) = \left\lbrace
\begin{array}{ll}
\{0\} & \text{if}~ j < 0 \\ E & \text{if}~ j \geq 0
\end{array}
\right. .$
\item
If $\phi_{\R}(\rho')= \rho \in \Sigma(1)$ and $\phi(u_{\rho'}) = b_{\rho} \, u_{\rho} \,$,
then $~\widetilde{E}^{\rho'}(j) = E^{\rho}\left(
\left\lfloor \frac{j}{b_{\rho}} \right\rfloor \right)$.
\end{enumerate}
\end{prop}

\begin{proof}
For $\sigma \in \Sigma$, we define $U_{\sigma} = \Spec( \C[S_{\sigma}])$ as an affine
open subset of $X$ and for $\sigma' \in \Sigma'$, we define
$U_{\sigma'}' = \Spec( \C[S_{\sigma'}])$ as an affine open subset of $X'$.
The sheaf $\pi^{\ast} \cE$ is defined by
\begin{equation}\label{eq:pullback-sheaf}
\pi^{\ast} \cE = \pi^{-1} \cE \otimes_{\pi^{-1} \cO_{X}} \cO_{X'}
\end{equation}
where for any sheaf $\cF$ on $X$, $\pi^{-1}\cF$ is defined by
$$
\Gamma(U', \, \pi^{-1} \cF) = \varinjlim_{U \supseteq \phi(U')} \Gamma(U, \, \cF) .
$$
We have $\Gamma(T, \, \cE) = E \otimes_{\C} \C[M]$. As $\pi(T') = T$, we deduce that
$$
\Gamma(T', \, \cE') = (E \otimes_{\C} \C[M]) \otimes_{\C[M]} \C[M'] \cong
E \otimes_{\C} \C[M'] ~.
$$
Thus, $\widetilde{E} = E$.

Let $\rho' \in \Sigma'(1)$ such that $\phi_{\R}(\rho') = \{0 \}$. By Lemma
\ref{lem:image-of-t-morphism} we get $\pi( O(\rho')) = T$. Hence, by the Orbit-Cone
Correspondence, $\pi(U_{\rho'}) = T$. By (\ref{eq:pullback-sheaf}), we deduce that
$$
\widetilde{E}^{\rho'} = \Gamma(U'_{\rho'}, \, \cE') =
(E \otimes_{\C} \C[M]) \otimes_{\C[M]} \C[S_{\rho'}] = E \otimes_{\C}
\C[S_{\rho'}] .
$$
If $m \in S_{\rho'}$, then $\widetilde{E}_{m}^{\rho'} = E$ and when $m \notin S_{\rho'}$,
then $\widetilde{E}_{m}^{\rho'} = \{0\}$. That is equivalent to say :
$\widetilde{E}^{\rho}(j) = E$ if $j \geq 0$ and $\widetilde{E}^{\rho}(j) = \{0\}$ if $j < 0$.

We now consider the case where $\phi_{\R}(\rho')= \rho $ and
$\phi(u_{\rho'}) = b_{\rho} \, u_{\rho}$. By Lemma \ref{lem:image-of-t-morphism}, we have
$\pi( O(\rho')) = O(\rho)$ and then $\pi(U'_{\rho'}) = U_{\rho}$. Hence,
$$
\widetilde{E}^{\rho'} = E^{\rho} \otimes_{\C[S_{\rho}]} \C[S_{\rho'}].
$$
As $\phi: N' \rightarrow N$ is surjective, there is an injective map $\psi: M \rightarrow M'$ such that for any $m \in M$ and $u' \in N'$, $\<m, \, \phi(u') \> = \< \psi(m), \, u' \>$.
Let $e_{\rho} \in M$ such that $M_{\R} = \R e_\rho \oplus \Span(u_{\rho})^\perp$ with
$\<e_\rho, u_\rho \> = 1$ and let
$e_{\rho'} \in M'$ such that $M'_{\R} = \R e_{\rho'} \oplus \Span(u_{\rho'})^\perp$ with 
$\<e_{\rho'}, u_{\rho'} \> = 1$. We set $M_0 = \Span(u_{\rho'})^\perp \cap M'$.
There is $m_{\rho} \in M_0$ such that
$$
\psi(e_\rho) = b_{\rho} \, e_{\rho'} + m_\rho .
$$
For any $m' \in M'$, there is $a \in \Z$ and $m_0 \in M_0$ such that
$m' = a e_{\rho'} + m_0$. Let $(a', \, r) \in \Z^2$ such that $a = a' b_{\rho} + r$ with
$0 \leq r < b_{\rho}$,
we have
$$
m' = r e_{\rho'} + a'(b_{\rho} e_{\rho'} + m_\rho) + (m_0 - a' m_\rho) = r e_{\rho'} +
\psi(a' e_\rho) + (m_0 - a' m_\rho) ~.
$$
Thus, $M' = A + \psi(M) + M_0$ where $A = \{ k \,e_{\rho'} : 0 \leq k \leq b_{\rho} - 1 \}$.
Therefore, $S_{\rho'} = A + \psi(S_{\rho}) + M_0$ and
$$
\C[S_{\rho'}] \cong \bigoplus_{m'' \in A} \C[S_{\rho}] \otimes_{\C}(\chi^{m''} \cdot \C[M_0])
$$
where for $m \in S_{\rho}$, $m' \in M_0$ and $m'' \in A$,
$\chi^m \otimes (\chi^{m''} \cdot \chi^{m_0}) = \chi^{m'' + \psi(m) + m_0}$. Thus,
\begin{align*}
\widetilde{E}^{\rho'} & \cong \bigoplus_{m'' \in A} E^{\rho} \otimes_{\C}
(\chi^{m''} \cdot \C[M_0])
\\ & =
\bigoplus_{m'' \in A} \left( \sum_{m \in M, \, m_0 \in M_0} E^{\rho}(\< m, \, u_\rho \>)
\otimes \chi^{m'' + \psi(m) + m_0}
\right) .
\end{align*}
As $\<m, \, u_{\rho} \> \in \Z$, for any $(m'', m_0) \in A \times M_0$,
$$
\<m, \, u_\rho \> = \dfrac{\< \psi(m), \, u_{\rho'} \>}{b_{\rho}} =
\dfrac{\< \psi(m) + m_0, \, u_{\rho'} \>}{b_{\rho}} =
\left\lfloor \dfrac{ \<m'' + \psi(m) + m_0, \, u_{\rho'} \>}{b_{\rho}} \right\rfloor ~.
$$
Thus,
\begin{align*}
\widetilde{E}^{\rho'} & \cong
\sum_{ \substack{m \in M, \, m_0 \in M_0, \\ m'' \in A}} E^{\rho}\left( \left\lfloor
\dfrac{\< m'' + \psi(m) + m_0, \, u_{\rho'} \>}{b_{\rho}} \right\rfloor \right) \otimes
\chi^{m'' + \psi(m) + m_0}
\\ & \cong
\sum_{m' \in M'} E^{\rho}\left( \left\lfloor
\dfrac{ \< m', \, u_{\rho'} \>}{b_{\rho}} \right\rfloor \right) \otimes \chi^{m'} \, ,
\end{align*}
that is $\widetilde{E}^{\rho'}(j) = E^{\rho}\left(
\left\lfloor \frac{j}{b_{\rho}} \right\rfloor \right)$.
\end{proof}

\begin{notation}
Let $F$ be a vector subspace of $E$. We denote by
$\left( F, \, \{ \widetilde{F}^{\rho'}(j) \} \right)$ the family of
filtrations of $(\pi^{\ast} \cE_F)^{\vee \vee}$.
\end{notation}

\begin{cor}\label{cor:pullback-sheaf-surjection}
Let $F$ be a vector subspace of $E$. Then the family of filtrations
$\left( F, \, \{ \widetilde{F}^{\rho'}(j) \} \right)$ satisfies
$\widetilde{F}^{\rho'}(j) = F \cap \widetilde{E}^{\rho'}(j)$ for all $\rho'$ such that
$\phi_{\R}(\rho') \in \{0 \} \cup \Sigma(1)$.
\end{cor}

\begin{proof}
If $\phi_{\R}(\rho') = \{0 \}$, we have
$\widetilde{F}^{\rho'}(j) = \left\lbrace
\begin{array}{ll}
\{0\} & \text{if}~ j < 0 \\ F & \text{if}~ j \geq 0
\end{array}
\right.$; so $\widetilde{F}^{\rho'}(j) = F \cap \widetilde{E}^{\rho'}(j)$.
\\
If $\phi_{\R}(\rho')= \rho \in \Sigma(1)$ and $\phi(u_{\rho'}) = b_{\rho} \, u_{\rho} \,$,
we have
$$
\widetilde{F}^{\rho'}(j) =
F^{\rho}\left( \left\lfloor \frac{j}{b_{\rho}} \right\rfloor \right) =
F \cap E^{\rho}\left( \left\lfloor \frac{j}{b_{\rho}} \right\rfloor \right) =
F \cap \widetilde{E}^{\rho'}(j) ~.
$$
\end{proof}

\subsection{Slopes of the pulled back sheaves}
\label{sec:slope-pullback-fibration}
We now assume that $\pi: X' \rightarrow X$ is a toric fibration between two complete and
$\Q$-factorial toric varieties. By \cite[Theorem 3.4.11]{CLS}, the toric morphism $\pi$
is proper. We set $n = \dim X$ and $r= \dim X' - \dim X$. Let $L$ be an ample divisor on
$X$ and $L'$ a $\pi$-ample divisor on $X'$. For $\ep \in \Q_{>0}$ small enough,
$L_{\ep} = \pi^{\ast}L + \ep L' \in \Pic(X') \otimes_{\Z} \Q$ defines an ample
$\Q$-divisor on $X'$.
In this section, we relate the slopes of sheaves on $X$ with respect to $L$ to the slopes
of their pullbacks on $X'$ with respect to $L_\ep$.
All intersection products are made in the Chow rings $A^{\bullet}(X)_\Q$ and
$A^{\bullet}(X')_\Q$ (cf. Proposition \ref{prop:intersection-toric-simplicial}).

\begin{prop}
Let $\cE$ be an equivariant reflexive sheaf on $X$ with family of filtrations given by
$\left(E, \, \{ E^{\rho}(j) \} \right)$ and $\cE' = (\pi^\ast \cE)^{\vee \vee}$.
Then, there is $C > 0$ such that
\begin{multline}\label{eq:slopepullback-fibration}
\mu_{L_{\ep}}(\cE') = C \, \mu_{L}(\cE) \, \ep^r - \dfrac{\ep^{r+1}}{\rk \cE}
\sum_{k=0}^{n-2} \dbinom{n+r-1}{k} \times
\\
\sum_{\rho \in \Sigma(1)} \iota_{\rho}(\cE) \, \ep^{n-k-2}
(\pi^{\ast} D_{\rho}) \cdot (\pi^{\ast} L)^k \cdot (L')^{n+r-k-1}
\end{multline}
and for any vector subspace $F$ of $E$,
\begin{multline}\label{eq:slopepullback-fibration-sat-sheaf}
\mu_{L_\ep}(\cE'_F)-\mu_{L_\ep}((\pi^\ast \cE_F)^{\vee \vee})=
- \dfrac{\ep^{r+1}}{\rk F} \times \\
\sum_{k=0}^{n-2} \binom{n+r-1}{k} \sum_{\rho' \in \Delta}
\left( \iota_{\rho'}(\cE'_{F}) - \iota_{\rho'}( (\pi^* \cE_F)^{\vee \vee} ) \right)
\ep^{n-k-2} D_{\rho'} \cdot (\pi^{\ast} L)^k \cdot (L')^{n+r-k-1}
\end{multline}
where
$\Delta = \{ \rho' \in \Sigma'(1): \phi_{\R}(\rho') \notin(\Sigma(0) \cup \Sigma(1)) \}$
and
$\iota_{\rho}(\cE)$ given in (\ref{eq:c1-sheaf}).
\end{prop}

\begin{rem}
The set $\Delta$ indexes the invariant divisors of $X'$ contracted by $\pi$.
\end{rem}

\begin{proof}
First, we have
$$
L_{\ep}^{n+r-1} = \sum_{k=0}^{n+r-1}\dbinom{n+r-1}{k} (\pi^{\ast}L)^{k} \cdot
(\ep L')^{n+r-k-1} .
$$
Let $D$ be a divisor on $X$. By the Projection formula (see
\cite[Proposition 2.3]{Fulton-intersection}),
for any $k \in \{0, \ldots, n+r-1 \}$,
$$
\pi_\ast \left( (\pi^\ast D) \cdot (\pi^{\ast} L)^k \cdot (L')^{n+r-k-1} \right) =
D \cdot L^k \cdot \pi_{\ast} ((L')^{n+r-k-1}) \in A_{0}(X).
$$
Hence,
$$
\deg \left( (\pi^\ast D) \cdot (\pi^{\ast} L)^k \cdot (L')^{n+r-k-1} \right) =
\deg \left( D \cdot L^k \cdot \pi_{\ast} ((L')^{n+r-k-1}) \right) ~.
$$
If $k \geq n$, then $D \cdot L^k \cdot \pi_{\ast} ((L')^{n+r-k-1}) = 0 \in A_0(X)$. Hence,
$$
(\pi^\ast D) \cdot (\pi^{\ast} L)^k \cdot (L')^{n+r-k-1} = 0.
$$
Since $L'$ is relatively ample, if $V \subseteq X'$ is an irreducible subvariety of
positive dimension that maps to a point in $X$, then $(L')^{\dim V} \cdot V > 0$
(cf. \cite[Corollary 1.7.9]{Laz-positivity1}). So in the case where $k=n-1$, one has
$$
(\pi^\ast D) \cdot (\pi^{\ast} L)^{n-1} \cdot (L')^r > 0 ~.
$$
As $\pi_{\ast} ((L')^r) \in A_{n}(X)$, we deduce that there is a constant $C>0$ such that
$\pi_{\ast} ((L')^r) = C \cdot [X]$. Thus,
$$
(\pi^{\ast}D) \cdot (\pi^{\ast}L)^{n-1} \cdot (L')^r =
C \, (D \cdot L^{n-1}) = C \, \deg_{L}(D) ~.
$$
Therefore, the degree of $\pi^{\ast}D$ with respect to $L_\ep$ is given by
\begin{align*}
\deg_{L_{\ep}}(\pi^\ast D) = & C \ep^{r} \, \deg_L(D) \\
& + \ep^{r+1} \sum_{k=0}^{n-2} \dbinom{n+r-1}{k} \ep^{n-k-2}
(\pi^{\ast} D) \cdot (\pi^{\ast} L)^k \cdot (L')^{n+r-k-1}  ~.
\end{align*}
As $\mu_{L_{\ep}}(\cE') = c_1(\cE') \cdot L_{\ep}^{n+r-1} =
\pi^{\ast}(c_1(\cE)) \cdot L_{\ep}^{n+r-1}$, 
according to Formulas (\ref{eq:c1-sheaf}) and (\ref{eq:slope-sheaf}), we get
Formula (\ref{eq:slopepullback-fibration}).

Let now $F$ be a vector subspace of $E$. We recall that $\cE'_F$ is the saturated
subsheaf of $\cE'$ associated to $F$ (cf. Notation \ref{notation:subsheaf}).
We wish to compare the slopes of $\cE_F'$ and of $(\pi^*\cE_F)^{\vee \vee}$.
We denote by $\left( F, \{ \widetilde{F}^{\rho'}(j) \} \right)$ the family of filtrations
of $(\pi^* \cE_F)^{\vee \vee}$.
By Corollary \ref{cor:pullback-sheaf-surjection}, for any $\rho' \in \Sigma'(1)$ such that
$\phi_{\R}(\rho') \in \{0\} \cup \Sigma(1)$,
$\widetilde{F}^{\rho'}(j) = \widetilde{E}^{\rho'}(j) \cap F$. Therefore, according to
(\ref{eq:c1-sheaf}), one has
\begin{equation}\label{eq:chern-pullback-fibration-sat-sheaf}
c_1(\cE'_{F}) = c_1( (\pi^* \cE_F)^{\vee \vee} ) - \sum_{\rho' \in \Delta}
\left( \iota_{\rho'}(\cE'_{F}) - \iota_{\rho'}( (\pi^* \cE_F)^{\vee \vee} ) \right)
D_{\rho'} \,.
\end{equation}
Let $\rho' \in \Delta$, then $\dim \pi( D_{\rho'}) \leq n-2$.
For $k \in \{0, \ldots, n+r-1\}$ ,
$$D_{\rho'} \cdot (L')^{n+r-k-1} \in A_{k}(|D_{\rho'}| \cap |(L')^{n+r-k-1}|)$$
and
$$
\pi_\ast \left( D_{\rho'} \cdot (L')^{n+r-k-1} \right) \in
A_{k}\left( \pi( |D_{\rho'}| \cap |(L')^{n+r-k-1}|) \right) \,.
$$
As $\dim \pi( |D_{\rho'}| \cap |(L')^{n+r-k-1}|) \leq n-2$, we deduce that
$\pi_\ast \left( D_{\rho'} \cdot (L')^{n+r-k-1} \right) =0$
if $k \geq n-1$. Thus,
$$
(\pi^{\ast}L)^{k} \cdot (\ep L')^{n+r-k-1} \cdot D_{\rho'} = 0
$$
if $k \geq n-1$. Therefore, for any $\rho' \in \Delta$,
$$
\deg_{L_\ep}(D_{\rho'}) = \ep^{r+1} \sum_{k=0}^{n-2} \dbinom{n+r-1}{k} \ep^{n-k-2}
D_{\rho'} \cdot (\pi^{\ast} L)^k \cdot (L')^{n+r-k-1}.
$$
Using (\ref{eq:chern-pullback-fibration-sat-sheaf}) and (\ref{eq:slope-sheaf}),
we get Formula (\ref{eq:slopepullback-fibration-sat-sheaf}).
\end{proof}

\subsection{Stability of the pulled back sheaf along a fibration}
\label{sec:stability-pullback-fibration}
We can now give the proofs of Theorem \ref{theo:stablecaseintro}, Theorem
\ref{theo:semistablecaseintro} and Proposition \ref{prop:unstablecaseintro}. We keep the
notations of the previous section.
Recall that to check slope stability of $\cE'$, by Proposition
\ref{prop:slope-with-saturated-sheaf} and Lemma \ref{lem:saturated-subsheaf}, it is enough
to compare slopes with subsheaves of the form $\cE'_F$. According to Formulas
(\ref{eq:slopepullback-fibration}) and (\ref{eq:slopepullback-fibration-sat-sheaf}), for
any vector subspace $F$ of $E$, we have
\begin{equation}\label{eq:slope-pullback-fib}
\mu_{L_\ep}(\cE') - \mu_{L_\ep}(\cE'_F) =
\left( \mu_{L_\ep}(\cE') - \mu_{L_\ep}( (\pi^* \cE_F)^{\vee \vee}) \right) +
\left( \mu_{L_\ep}((\pi^* \cE_F)^{\vee \vee}) - \mu_{L_\ep}(\cE'_F) \right)
\end{equation}
where
$$
\left\lbrace
\begin{array}{l}
\mu_{L_\ep}((\pi^* \cE_F)^{\vee \vee}) - \mu_{L_\ep}(\cE'_F) = o(\ep^{r})
\\
\mu_{L_\ep}(\cE') - \mu_{L_\ep}((\pi^* \cE_F)^{\vee \vee}) =
C(\mu_L(\cE) - \mu_L(\cE_F)) \ep^r + o(\ep^{r})
\end{array}
\right. ~.
$$
As $\cE'_F$ is the saturation of $(\pi^\ast \cE_F)^{\vee \vee}$, we have
$\mu_{L_\ep}((\pi^\ast \cE_F)^{\vee \vee}) - \mu_{L_\ep}(\cE'_F) \leq 0$.

\begin{proof}[Proof of Theorem \ref{theo:stablecaseintro}]
We assume that $\cE$ is stable with respect to $L$. For any vector subspace $F$ of $E$,
we have $\mu_{L}(\cE) - \mu_{L}(\cE_F)>0$. We set
$$
a_0 = \min \{\mu_{L}(\cE) - \mu_{L}(\cE_F): \{0\} \subsetneq F \subsetneq E \} ~.
$$
By Lemma \ref{lem:slope-number-vector-space}, one has $a_0 >0$. As the set
$\{\mu_{L_\ep}(\cE') - \mu_{L_\ep}(\cE'_F): \{0\} \subsetneq F \subsetneq E \}$ is
finite, we deduce that the number of vector spaces $F$ to consider is finite.
By Equation (\ref{eq:slope-pullback-fib}), we get
$$
\mu_{L_\ep}(\cE') - \mu_{L_\ep}(\cE'_F) \geq C a_0 \, \ep^r + o(\ep^r) ~.
$$
Thus, there is $\ep_0 >0$, such that for any $\ep \in (0 , \ep_0) \cap \Q$,
$\mu_{L_{\ep}}(\cE') - \mu_{L_{\ep}}(\cE'_F) > 0$.
Hence, we deduce that $\cE'$ is stable with respect to $L_{\ep}$.
\end{proof}

\begin{proof}[Proof of Proposition \ref{prop:unstablecaseintro}]
We assume that $\cE$ is unstable with respect to $L$. There is a vector subspace $F$ of $E$
with $0< \dim F < \dim E$ such that $\mu_{L}(\cE) - \mu_{L}(\cE_F)<0$.
By (\ref{eq:slope-pullback-fib}), there is $\ep_0 >0$, such that for any
$\ep \in (0 , \ep_0) \cap \Q$, $\mu_{L_{\ep}}(\cE') - \mu_{L_{\ep}}(\cE'_F) < 0$.
Hence, $\cE'$ is unstable with respect to $L_{\ep}$.
\end{proof}

\begin{rem}
Theorem \ref{theo:stablecaseintro} and Proposition \ref{prop:unstablecaseintro}
also follow from the openness property of stability \cite[Theorem 3.3]{GKPeternell15}.
\end{rem}

We now consider the case where $\cE$ is an equivariant strictly semistable sheaf on $(X,L)$.
There is a Jordan-H\"older filtration
$$
0=\cE_1 \subseteq \cE_2 \subseteq \ldots \subseteq \cE_\ell=\cE
$$
by slope semistable subsheaves with stable quotients of same slope as $\cE$. We denote by
$\Gr(\cE):=\bigoplus_{i=1}^{\ell-1}\cE_{i+1}/\cE_i$ the graded object of $\cE$ and
$\mathfrak{E}$ the set of equivariant and saturated reflexive subsheaves
$\cF \subseteq \cE$ arising in a Jordan-H\"older filtration of $\cE$.
Before proving Theorem \ref{theo:semistablecaseintro}, we need an extra auxiliary lemma.

\begin{lem}\label{lem:pullback-saturated-sheaf}
Let $\pi: X' \rightarrow X$ be a fibration and $\cE$ be a locally free sheaf on $X$.
If $\cF$ is a coherent subsheaf of $\cE$ such that $\cE / \cF$ is locally free, then
$\pi^\ast \cF$ is a saturated subsheaf of $\pi^\ast \cE$.
\end{lem}

\begin{proof}
Let $\cG$ be the quotient sheaf of $\cE$ by $\cF$.
As we have an exact sequence
$0 \imp \pi^{-1} \cF \imp \pi^{-1} \cE \imp \pi^{-1} \cG \imp 0$ and
$\pi^\ast \cE = \pi^{-1} \cE \otimes_{\pi^{-1} \cO_X} \cO_{X'}$ we get
$$
\Tor_{1}^{\pi^{-1} \cO_{X}}(\pi^{-1} \cG, \, \cO_{X'}) \imp \pi^{\ast} \cF \imp
\pi^\ast \cE \imp \pi^\ast \cG \imp 0 ~.
$$
As $\cG$ is a locally free $\cO_X$-module, we deduce that $\pi^{-1} \cG$ is a locally
free $\pi^{-1} \cO_X$-module, therefore
$\Tor_{1}^{\pi^{-1}\cO_X}(\pi^{-1} \cG, \, \cO_{X'}) = 0$. Hence, we have an exact
sequence
$$
0 \imp \pi^{\ast} \cF \imp \pi^\ast \cE \imp \pi^\ast \cG \imp 0 ~.
$$
As $\pi^\ast \cG \simeq \pi^\ast \cE / \pi^\ast \cF$ and $\pi^\ast \cG$ is locally free,
we deduce that $\pi^\ast \cE / \pi^\ast \cF$ is torsion free. Hence, $\pi^\ast \cF$ is
saturated in $\pi^\ast \cE$.
\end{proof}

\begin{proof}[Proof of Theorem \ref{theo:semistablecaseintro}]
Let $\fF = \{ F \subsetneq E : \mu_{L}(\cE_F) < \mu_{L}(\cE) \}$.
By Equation (\ref{eq:slope-pullback-fib}), for any $F \in \fF$, there is $\ep_F >0$ such
that for any $\ep \in (0, \ep_F) \cap \Q$, $\mu_{L_\ep}(\cE'_F) < \mu_{L_\ep}(\cE')$.
We set
$$
\ep_1 = \min \{\ep_F : F \in \fF \}.
$$
As by Lemma \ref{lem:slope-number-vector-space} it suffices to compare slopes for a finite
set of vector subspaces, we deduce that $\ep_1 >0$. Thus, the subsheaves $\cE'_F$ for
$F \in \fF$ will never destabilize $\cE'$ for $\ep < \ep_1$.

We then consider $F\notin \fF$, that is the case where $\mu_{L}(\cE_F) = \mu_{L}(\cE)$.
We then have by definition $\cE_F \in \fE$. As $\cE$ is locally free and $\Gr(\cE)$ is
sufficiently smooth, by Lemma \ref{lem:pullback-saturated-sheaf}, $\pi^\ast \cE_F$ is
saturated in $\cE'$. Hence, $(\pi^\ast \cE_F)^{\vee \vee}=\cE'_F$ and
\begin{equation}\label{eq:slope-saturated-and-pullback}
\mu_{L_\ep}((\pi^\ast \cE_F)^{\vee \vee}) - \mu_{L_\ep}(\cE'_F) = 0.
\end{equation}
Therefore, for any $F \subseteq E$ such that $\cE_F \in \fE$,
$$
\mu_{L_\ep}(\cE') - \mu_{L_\ep}(\cE'_F) = \mu_{L_\ep}(\cE') -
\mu_{L_\ep}((\pi^\ast \cE_F)^{\vee \vee}).
$$
But then the sign of $\mu_{L_\ep}(\cE') - \mu_{L_\ep}(\cE'_F)$ is given by the sign of
the coefficient $\mu_0(\cE') - \mu_0(\cE'_F)$ of the smallest exponent in the expansion
in $\ep$ of $\mu_{L_\ep}(\cE')-\mu_{L_\ep}(\cE'_F)$. Again, as we only need to test for a
finite number of subspaces $F\subseteq E$, we obtain the result, with $\ep_0\leq \ep_1$.
\end{proof}

\subsection{The case of locally trivial fibrations}
\label{sec:loctrivialfibration}
We assume here that $\pi: X' \to X$ is locally trivial. We use the notations of Section \ref{sec:fibration}.
Let $\cE$ be an equivariant reflexive sheaf on $X$ given by the family of filtrations
$\left( E, \, \{ E^{\rho}(j) \} \right)$.
As for any $\rho' \in \Sigma'(1)$, $\phi_{\R}(\rho') \in \Sigma(0) \cup \Sigma(1)$, by
Corollary \ref{cor:pullback-sheaf-surjection} one has
$(\pi^\ast \cE_F)^{\vee \vee} = \cE'_F$ for any vector subspace $F$ of $E$.
According to (\ref{eq:slope-pullback-fib}), we have
\begin{equation}\label{eq:slople-trivial-fib}
\mu_{L_\ep}(\cE') - \mu_{L_\ep}(\cE'_F) =
\mu_{L_\ep}(\cE') - \mu_{L_\ep}((\pi^\ast \cE_F)^{\vee \vee}) ~.
\end{equation}
Therefore, in the proof of Theorem \ref{theo:semistablecaseintro}, identity
(\ref{eq:slope-saturated-and-pullback}) holds for any vector subspace $F$ of $E$.
Hence, in the case of locally trivial fibration, the assumptions on $\cE$ and $\Gr(\cE)$
to be locally free in Theorem \ref{theo:semistablecaseintro} are not necessary. Let's now consider a simple example to illustrate our results. We will assume that $X'=X$, so that the only perturbation we consider is in the polarisation from $L$ to $L_\ep'$.
\begin{example}
\label{ex:exampleF2}
Let $(e_1, e_2)$ be a basis of $\Z^2$. We set $u_1=e_1$, $u_2=e_2$, $u_3=e_2-2e_1$ and
$u_4= -e_2$. Let $X$ be the singular toric surface associated to the fan
$$
\Sigma = \{0\} \cup \{ \Cone(u_i): 1 \leq i \leq 4 \} \cup
\{ \Cone(u_i, u_{i+1}): 1 \leq i \leq 4 \}.
$$
We denote by $D_i$ the divisor corresponding to the ray $\Cone(u_i)$.
As $\Sigma$ is simplicial, the divisors $D_i$ are $\Q$-Cartier. There are linear
equivalences $D_1 \sim_{\rm lin} 2 D_3$ and $D_2 \sim_{\rm lin} D_4 - D_3$. According to
Lemma \ref{lem:intersection-toric-simplicial}, we have
$$
D_3 \cdot D_4 = \dfrac{1}{2} \qquad D_3 \cdot D_2 = \dfrac{1}{2} \qquad D_3 \cdot D_3 =0
\qquad D_4 \cdot D_1 = 1 \qquad D_4 \cdot D_4 = \dfrac{1}{2} ~.
$$
Hence the divisor $a D_3 + b D_4$ is ample if and only if $a, \, b >0$.
As $-K_X = D_1 + D_2 + D_3 + D_4 \sim_{\rm lin} 2(D_3 + D_4)$, we deduce that $X$ is a Fano
surface. Let $\cE$ be the tangent bundle of $X$ (see Example \ref{examp:filtration-tangent}
for its family of filtrations).
If $L$ is an ample line bundle, to check the stability of $\cE$ with respect to $L$, it
suffices to compare $\mu_{L}(\cE)$ with $\mu_{L}(\cE_F)$ for
$F \in \{F_1, F_2, F_3 \}$ where $F_1 =\Span(u_1)$, $F_2 =\Span(u_2)$ and $F_3 =\Span(u_3)$.
We assume that $L= - K_X$. We have
$$
L \cdot D_1 = 2 \qquad L \cdot D_2 = 1 \qquad L \cdot D_3 = 1 \qquad
L \cdot D_4 = 2
$$
and
$$
\mu_{L}(\cE) = 3 \qquad \mu_{L}(\cE_{F_1})=2 \qquad \mu_{L}(\cE_{F_2})=3 \qquad
\mu_{L}(\cE_{F_3})=1 ~.
$$
Hence $\cE$ is strictly semistable with respect to $-K_X$.

We now consider $L'_\ep = L + \ep(a D_3 + b D_4)$. From our criterion, to check stability
of $\cE$ with respect to $L'_\ep$, it is enough to compare slopes of $\cE$ and $\cE_{F_2}$.
We have
$$
L'_\ep \cdot D_1 = 2 + b \ep, \qquad L'_\ep \cdot D_2 = 1 + \frac{a \ep}{2}, \qquad
L'_\ep \cdot D_3 = 1 + \frac{b \ep}{2}, \qquad L'_\ep \cdot D_4 = 2 + \frac{(a+b)\ep}{2}.
$$
Thus,
$\mu_{L'_\ep}(\cE) = 3 + \left(b + \frac{a}{2} \right) \ep$, and
$\mu_{L'_\ep}(\cE_{F_2})=3 + \left(a + \frac{b}{2} \right) \ep$.
We deduce that $\cE$ is stable (resp. strictly semistable) with respect to $L'_\ep$ if and
only if $b-a >0$ (resp. $b-a =0$).
\end{example}

\subsection{Stability of sheaves in family}
\label{sec:families}
Let $S$ be a scheme of finite type over $\C$, and let $T_S$ be the relative torus of
$X \times S \rightarrow S$. An $S$-family of equivariant reflexive sheaves on $X$ is a
reflexive sheaf $\cE$ on $X \times S$ with an action of the relative torus $T_S$
compatible with the action on $X \times S$.
For $\rho \in \Sigma(1)$, we define an order relation $\preceq_{\rho}$ on $M$ by setting
$m \preceq_{\rho} m'$ if and only if $m' - m \in S_{\rho}:= \rho^{\vee} \cap M$. We write
$m \prec_{\rho} m'$ if we have $m \preceq_{\rho} m'$ but not $m' \preceq_{\rho} m$.
Following \cite[Proposition 3.13]{payne2008} and \cite[Proposition 3.4]{Koo11}, we have:

\begin{prop}
The category of $S$-families of equivariant reflexive sheaves on $X$ is equivalent to
the category of reflexive sheaves $\cF$ on $S$ with collections of increasing filtrations
$$
\{ \cF_{m}^{\rho}: m \in M \}_{\rho \in \Sigma(1)}
$$
indexed by the rays of $\Sigma$ having the following properties:
\begin{enumerate}
\item
for all $m, m' \in M$ with $m \preceq_\rho m'$, there are injections
$\cF_{m}^\rho \hookrightarrow \cF_{m'}^\rho$ and $\cF_{m}^\rho \hookrightarrow \cF$;
\item
for each chain $\cdots \prec_\rho m_{i-1} \prec_\rho m_i \prec_\rho \cdots$ of
elements of $M$, there exists $i_0 \in \Z$ such that $\cF_{m_i}^\rho = 0$ for all
$i \leq i_0$;
\item
and, there are only finetely many $m \in M$ such that the morphism
$$
\bigoplus_{m' \prec_\rho m} \cF_{m'}^\rho \longrightarrow \cF_{m}^\rho
$$
is not surjective.
\end{enumerate}
\end{prop}

Let $\cE = (\cE_s)_{s \in S}$ be an $S$-family of equivariant reflexive sheaves on $X$
given by the collection of increasing filtrations $(\cF, \{ \cF_{m}^{\rho}: m \in M \} )$.
Then, for any $s \in S$, the family of filtrations $(E_s, \{ E_{s}^{\rho}(j) \})$ of the
reflexive sheaf $\cE_s$ is given by
$$
E_s = \cF(s) \quad \text{and} \quad E_{s}^{\rho}(j) = \cF_{m}^{\rho}(s)
$$
with $m \in M$ such that $\<m, u_\rho \> = j$ where $\cF(s)$ and $\cF_{m}^{\rho}(s)$ are
respectively the fiber of $\cF$ and $\cF_{m}^{\rho}$ at $s$.
We first observe that:

\begin{lem}\label{lem:slope-sheaf-family}
Fix an ample divisor $L$ on $X$. If for all $\rho \in \Sigma(1)$ and $j \in \Z$, the map
$s \mapsto \dim(E_{s}^\rho(j))$ is constant, then the set
$$
\{ \mu_{L}( (\cE_s)_F ) : s \in S, \; 0 \subsetneq F \subsetneq E_s \}
$$
is finite.
\end{lem}

\begin{proof}
The proof is similar to the proof of Lemma \ref{lem:slope-number-vector-space}. For any
$\rho \in \Sigma(1)$, there is $(j_{\rho}, J_\rho) \in \Z^2$ such that for any $s \in S$,
$E_{s}^{\rho}(j) = \{0\}$ if $j < j_\rho$ and $E_{s}^{\rho}(j)= E_s$ if $j \geq J_{\rho}$.
As
\begin{multline*}
\{\dim(E_{s}^{\rho}(j) \cap F) - \dim(E_{s}^{\rho}(j-1) \cap F): j \in \Z, s \in S, \;
\text{and} \; 0 \subsetneq F \subsetneq E_s \} \\
\subseteq \{0, 1, \ldots, \rk(\cE) \} \, ,
\end{multline*}
we deduce that $\{ \iota_{\rho}( (\cE_s)_F) : s \in S, 0 \subsetneq F \subseteq E_s \}$ is
finite. Hence, the lemma follows from Formula (\ref{eq:slope-sheaf}).
\end{proof}

The previous lemma is the key to obtain a family version of Theorems
\ref{theo:stablecaseintro} and \ref{theo:semistablecaseintro}.
The {\it characteristic function} $\chi$ of an equivariant reflexive sheaf $\cG$ with
family of filtrations $(F, \{F^{\rho}(j) \})$ is the function
$$
\begin{array}{cccc}
\chi(\cG) : & M & \imp & \Z^{\sharp \Sigma(n)} \\
& m & \longmapsto & (\dim(F_{m}^{\sigma}))_{\sigma \in \Sigma(n)}
\end{array}
$$
where $F_{m}^\sigma = \bigcap_{\rho \in \sigma(1)} F^{\rho}(\<m, u_\rho \>)$.
The families that we will consider will satisfy one of the following:
\begin{enumerate}[label=(\Roman*)]
\item
$\cE$ is locally free on $X \times S$, or
\label{item:pullback-family1}
\item
the characteristic function $(\chi(\cE_s))_{s \in S}$ is constant.
\label{item:pullback-family2}
\end{enumerate}

\begin{lem}\label{lem:dimconstant}
Let $X$ be a smooth toric variety. Assume that $(\cE_s)_{s \in S}$ satisfies
\ref{item:pullback-family1} or \ref{item:pullback-family2}. Then for all
$\rho \in \Sigma(1)$ and $j \in \Z$, $s \mapsto \dim(E_{s}^\rho(j))$ is constant.
\end{lem}

\begin{proof}
In the case that the family satisfies \ref{item:pullback-family1}, by
\cite[Proposition 3.13]{payne2008} (Klyachko's compatibility condition for
$S$-families of locally free sheaves), for any $\sigma \in \Sigma(n)$, there is a multiset
$A_\sigma \subseteq M$ of size $\rk(\cE)$ such that for any $m \in M$, $\cF_{m}^{\rho}$
is a locally free sheaf of rank
$$
\left| \{ \alpha \in A_\sigma : \< \alpha, u_\rho \> \leq \<m, u_\rho \> \} \right| .
$$
As for any $s \in S$ and $m \in M$, $\dim( \cF_{m}^{\rho}(s) ) = \rk( \cF_{m}^{\rho} )$,
we deduce that the map
$$
s \longmapsto \dim(E_{s}^{\rho}(\< m, u_\rho \>))
$$
is constant.

We now assume \ref{item:pullback-family2}. For any $\sigma \in \Sigma(n)$, the set
$\lbrace u_\rho : \rho \in \sigma(1) \rbrace$ is a basis of $N$. Then, for any
$\rho \in \sigma(1)$ and any $j \in \Z$, we can find an element $m \in M$ such that for
all $s \in S$,
$$
\langle m , u_{\rho} \rangle = j
$$
and for $\rho' \in \sigma(1) \setminus \lbrace \rho \rbrace$,
$$
E_{s}^{\rho'}(\langle m, u_{\rho'} \rangle ) = E_s.
$$
This can be made uniformly in $s$ as follows: we can fix
$m' \in M$ such that $E_{s,m'}^\sigma = E_s$ for all $s \in S$. This implies that for
$\rho' \in \sigma(1)$, $E_{s}^{\rho'}(\langle m', u_{\rho'} \rangle)= E_s$. Then, define
$$
m= j u_{\rho}^* + \sum_{\rho' \neq \rho} \langle m', u_{\rho'} \rangle u_{\rho'}^*
$$
where $\lbrace u_{\rho'}^* : \rho' \in \sigma(1) \rbrace$ is the dual basis of
$\lbrace  u_{\rho'} : \rho' \in \sigma(1) \rbrace$. But then
$$
E_{s,m}^\sigma= \bigcap_{\rho' \in \sigma(1)} E_{s}^{\rho'}(\langle m, u_{\rho'} \rangle)=
E_{s}^{\rho}(j)
$$
and \ref{item:pullback-family2} implies the result.
\end{proof}

Let $\pi : X' \rightarrow X$ be a toric fibration between two smooth toric varieties and
$(\cE_s)_{s \in S}$ be a family of reflexive sheaves on $X$ satisfying
\ref{item:pullback-family1} or \ref{item:pullback-family2}. Assume that for all $s \in S$,
$\cE_s$ is stable on $(X,L)$. From Lemma \ref{lem:dimconstant} and Lemma
\ref{lem:slope-sheaf-family}, we deduce that the $\ep_0$ in the proof of Theorem
\ref{theo:stablecaseintro} can be taken uniformly in $s \in S$.
Note for this that in the expansions in $\ep$ of formula (\ref{eq:slope-sheaf}) for the
slopes $\mu_{L_\ep}(\cE')$, the terms $\iota_\rho(\cE)$ do not vary with $\ep$, only the
terms $\deg_{L_\ep}(D_\rho)$ do. Similarly, we can take $\ep_0$ uniformly in Theorem
\ref{theo:semistablecaseintro} if all $\cE_s$ are assumed to be sufficiently smooth
on $(X,L)$.

We deduce from this the existence of injective maps between components of the moduli
spaces of stable equivariant reflexive sheaves on $(X,L)$ and on $(X', L_\ep)$, for $\ep$
small enough.
One can consider the moduli space of equivariant stable reflexive sheaves on $(X,L)$ with
fixed characteristic function  $\chi$ introduced in \cite{Koo11}, denoted
$\cN^{\mu s}_{\chi}(X,L)$. As $\chi$ determines the Chern character (\cite{Kly90} and
\cite[Section 3.4]{Koo11}), and thus the Hilbert polynomial by Hirzebruch-Riemann-Roch,
we deduce that the reflexive pullback induces an injective map for $0< \ep \ll 1$:
$$
\pi^* : \cN^{\mu s}_{\chi}(X,L) \longrightarrow \cN^{\mu s}_{P'}(X',L_\ep)
$$
where $P'$ denotes the Hilbert polynomial with respect to $L_\ep$ of any element
$(\pi^*\cE)^{\vee\vee}$ with characteristic function $\chi$. In fact, if we denote
$P_\chi$ the Hilbert polynomial induced by $\chi$, we expect that this map is actually
defined on 
$$
\cN^{\mu s}_{P}(X,L)= \bigcup_{P_\chi=P} \cN^{\mu s}_{\chi}(X,L)
$$
the moduli of stable equivariant reflexive sheaves with Hilbert polynomial $P$. In the
same way, fixing the total Chern class, one should obtain maps between the moduli spaces
of equivariant and stable locally free sheaves. Those spaces should be obtained as open
sub-schemes of the moduli spaces constructed in \cite{payne2008}. We believe that those
maps deserve further study and will come back to them in future research.

\section{Blow-ups}\label{sec:blowups}
In this section we specialize to equivariant blow-ups along smooth centers.
Let $X$ be a smooth toric variety of dimension $n$ associated to a smooth fan $\Sigma$. We denote
by $\pi: X' \rightarrow X$ the blowup of $X$ along $Z=V(\tau)$ with $\tau \in \Sigma$ such that
$\dim \tau \geq 2$ and we set $\Sigma' = \Sigma^\ast(\tau)$. As before, $\cE$ stands for an equivariant reflexive sheaf on $X$ and $\cE'$ denotes its reflexive pullback along $\pi$.

\subsection{Slope of the reflexive pullback along a blowup}
\label{sec:slope-pullback-blowup}
In this section, we derive formula (\ref{eq:formulaintro}), which is the key to the results
stated in the introduction. Note that the proof doesn't require any equivariant assumption.

\begin{prop}\label{prop:slope-pullback-blowup}
Let $X$ be a smooth projective variety and $Z \subseteq X$ a smooth irreducible
subvariety of dimension $\l$ with $1 \leq \l \leq \dim(X)-2$. We denote by
$\pi: X' \rightarrow X$ the blowup of $X$ along $Z$ and $D_0$ the exceptional divisor of
$\pi$.
Let $L$ be an ample divisor of $X$ and let $L_\ep = \pi^\ast L - \ep D_0$ be an ample
$\Q$-divisor of $X'$ for $\ep \in \Q_{>0}$ small. Then for any divisor $D$ of $X$,
\begin{equation}\label{eq:slope-pullback-blowup}
\pi^\ast D \cdot L_{\ep}^{n-1} = D \cdot L^{n-1} -
\dbinom{n-1}{\l-1} \ep^{n-\l} D \cdot L^{\l-1} \cdot Z + o(\ep^{n-\l})
\end{equation}
and
\begin{equation}\label{eq:degree-exceptional-div}
D_0 \cdot L_{\ep}^{n-1} = \dbinom{n-1}{\l} \ep^{n-\l -1} Z \cdot L^{\l} +
o(\ep^{n-\l-1}).
\end{equation}
\end{prop}

\begin{proof}
We denote by $\cN$ the normal bundle of $Z$. We have $D_0 = \P(\cN)$.
For a divisor $D$ of $X$, one has
\begin{align*}
\pi^\ast D \cdot L_{\ep}^{n-1} &=
\sum_{k=0}^{n-1} \dbinom{n-1}{k} \pi^\ast D \cdot (\pi^\ast L)^k \cdot (- \ep D_0)^{n-1-k}
\\ &=
\pi^\ast D \cdot (\pi^\ast L)^{n-1} + \sum_{k=0}^{n-2} \dbinom{n-1}{k} \pi^\ast D \cdot
(\pi^\ast L)^{k} \cdot (- \ep D_0)^{n-1-k}.
\end{align*}
Therefore, by the projection formula, we get
$$
\pi^\ast D \cdot L_{\ep}^{n-1} = D \cdot L^{n-1} + \sum_{k=0}^{n-2} \dbinom{n-1}{k}
(-\ep)^{n-1-k} D \cdot L^{k} \cdot \pi_\ast(D_{0}^{n-1-k})
$$
and also
$$
D_0 \cdot L_{\ep}^{n-1} = \sum_{k=0}^{n-1} \dbinom{n-1}{k} (-\ep)^{n-1-k} L^k \cdot
\pi_{*}(D_{0}^{n-k}).
$$
If $\eta = \pi_{|D_0}$, according to \cite[Example 3.3.4]{Fulton-intersection}, one has
$$
\sum_{k \geq 1}(-1)^{k-1} \eta_\ast(D_{0}^k) = s(\cN) \cap [Z]
$$
where $s(\cN)$ is the total Segre class of $\cN$. As $s_i(\cN) \cap [Z] \in A_{\l-i}(Z)$
and $\eta_\ast(D_{0}^{n-k}) \in A_{k}(Z)$, we deduce that
$$
(-1)^{n-k-1} \eta_{\ast}(D_{0}^{n-k}) =
\left\lbrace
\begin{array}{cl}
s_{\l-k}(\cN) \cap [Z] & \text{if $0 \leq k \leq \l$}
\\
0 & \text{if $\l+1 \leq k \leq n-1$}
\end{array}
\right..
$$
As $s_0(\cN) \cap [Z] = [Z]$, we get
$$
\begin{aligned}
\pi^\ast D \cdot L_{\ep}^{n-1} = D \cdot L^{n-1} & - \dbinom{n-1}{\l-1} \ep^{n-\l} D \cdot
L^{\l-1} \cdot Z \\ & -
\sum_{k=0}^{\l-2} \dbinom{n-1}{k} \ep^{n-1-k} D \cdot L^k \cdot \left(s_{\l-1-k}(\cN) \cap
[Z] \right)
\end{aligned}
$$
which gives the first formula. The second formula is obtained in the same way.
\end{proof}

\begin{cor}\label{cor:slope-pullback-blowup}
With the same data as Proposition \ref{prop:slope-pullback-blowup}, if $\cE$ is a
reflexive sheaf on $X$, then
$$
\mu_{L_\ep}((\pi^*\cE)^{\vee \vee}) = \mu_{L}(\cE) - \dbinom{n-1}{\l-1}
\mu_{L_{\vert Z}}(\cE_{|Z}) \ep^{n-\l} + O(\ep^{n-\l+1}) .
$$
\end{cor}

\subsection{Reflexive pullback along an equivariant blow-up}
\label{sec:pullbackandfiltrations}
In this section we give the family of filtrations of the reflexive pullback along an
equivariant blow-up. This will serve in relating the Chern classes, and also in obtaining explicit examples. Let $(u_1, \ldots, \, u_n)$ be a basis of $N$ such that
$\tau = \Cone(u_1, \ldots, \, u_s)$ with $2 \leq s \leq n$ and
$\left\lbrace \Cone(A): A \subseteq \{u_1, \ldots, \, u_n\} \right\rbrace \subseteq
\Sigma$.
We set $\rho_i = \Cone(u_i)$ for $i \in \{1, \ldots, \, s \}$ and
$\rho_0 = \Cone(u_\tau)$ where $u_{\tau} = u_1 + \ldots + u_s$.
We denote by $(e_1, \ldots, \, e_n)$ the dual basis of $(u_1, \ldots, \, u_n)$.

\begin{prop}\label{prop:filt-toric-blowup}
Let $\cE$ be an equivariant reflexive sheaf on $X$ given by the family of filtrations
$\left( E, \, \{ E^{\rho}(j) \} \right)$. Let
$\left( E, \, \{ \widetilde{E}^{\rho}(j) \}_{\rho \in \Sigma'(1), \,
j \in \Z} \right)$ be the family of filtrations of $\cE'=(\pi^\ast \cE)^{\vee \vee}$.
Then we have:
\begin{enumerate}
\item
if $\rho \in \Sigma(1) \subseteq\Sigma'(1)$, then $\widetilde{E}^{\rho}(j) = E^{\rho}(j)$;
\item
if $\rho = \rho_0$, then
$$
\widetilde{E}^{\rho}(j) = \sum_{i_1 + \ldots + i_s=j}
{E^{\rho_1}(i_1) \cap \ldots \cap E^{\rho_s}(i_s)} ~.
$$
\end{enumerate}
\end{prop}
Note that as $\pi$ is a toric morphism and $\cE$ is equivariant, the sheaf $\cE'$ is equivariant as well and admits a description in terms of a family of filtrations. 
\begin{proof}
The sheaf $\pi^*\cE$ is equivariant, and as such it will be enough to describe its spaces of sections on invariant subsets of $X'$ (cf  \cite[Proposition 5.7]{Per04}).

We recall that the $\Z$-linear map $\phi = \Id_N$ is compatible with $\Sigma'$ and
$\Sigma$.
If $\rho \in \Sigma(1) \subseteq\Sigma'(1)$, we have $\phi(u_\rho) = u_\rho$. By Proposition
\ref{prop:pullback-sheaf-surjection} we get $\widetilde{E}^{\rho}(j) = E^{\rho}(j)$.

We now assume that $\rho = \Cone(u_{\tau})$. The minimal cone of $\Sigma$ which contains
$\phi_{\R}(\rho)$ is $\tau$. As $\pi(T') = T$, by Lemma \ref{lem:image-of-t-morphism}, we
deduce that $\pi ( O(\rho) ) = O(\tau)$. Thus,
$\pi ( U_{\rho}' ) = T \cup O(\tau) \subsetneq U_\tau$.  One has
$$
\widetilde{E}^{\rho} =
\Gamma(U_{\rho}', \, \cE') \cong \Gamma(U_{\tau}, \cE)
\otimes_{\cO_{X}(U_{\tau})} \cO_{X'}(U_{\rho}') = E^{\tau} \otimes_{\cO_{X}(U_{\tau})}
\cO_{X'}(U_{\rho}')
$$
where $\cO_{X'}(U_{\rho}') = \C[S_{\rho}]$, $\cO_{X}(U_{\tau}) = \C[S_{\tau}]$ and
$E^\tau$ defined in Notation \ref{nota:decomposition-of-epuiv-sheaf}.
We have
$$
\tau^{\vee} = \Cone(e_1, \ldots, \, e_s, \, \pm e_{s+1}, \ldots, \, \pm e_n) ~.
$$
A point $m = m_1 \, e_1 + \ldots + m_n \, e_n$ lies in $\rho^{\vee}$ if and only if
$m_1 + \ldots + m_s \geq 0$, i.e $m_s \geq -(m_1 + \ldots + m_{s-1})$. Hence,
$$
\rho^{\vee} = \Cone(\pm(e_1 - e_s), \ldots, \, \pm(e_{s-1} - e_s), \, e_s, \,
\pm e_{s+1}, \ldots, \, \pm e_n )
$$
and
$$
\rho^{\perp} = \Cone(\pm(e_1 - e_s), \ldots, \, \pm(e_{n-1} - e_s), \,
\pm e_{s+1}, \ldots, \, \pm e_n) ~.
$$
Therefore, $S_{\rho} = \rho^{\perp} + S_{\tau}$ and
$\C[S_{\rho}] = \C[S_{\tau}] \otimes_{\C} \C[M(\rho)]$. Thus,
$$
\widetilde{E}^{\rho} = E^{\tau} \otimes_{\C[S_{\tau}]} \left( \C[S_{\tau}]
\otimes_{\C} \C[M(\rho)] \right) = E^{\tau} \otimes_{\C} \C[M(\rho)] ~.
$$
Hence,
\begin{align*}
\widetilde{E}^{\rho} = \sum_{m' \in M(\rho)}{E^{\tau} \otimes \chi^{m'}}
& = \sum_{m' \in M(\rho)} \left( \sum_{m \in M} E_{m}^{\tau} \otimes
\chi^m \right) \otimes \chi^{m'} 
\\ & =
\sum_{m' \in M(\rho)} \left( \sum_{m \in M} E_{m-m'}^{\tau} \otimes
\chi^m \right)
\\ & =
\sum_{m \in M} \left( \sum_{m' \in M(\rho)} E_{m-m'}^{\tau} \right) \otimes \chi^m
\end{align*}
Therefore, for any $m \in M$,
$$
\widetilde{E}_{m}^{\rho} = \sum_{m' \in M(\rho)} E_{m-m'}^{\tau}.
$$
As for any $m' \in M(\rho)$,
$ \<m-m', u_1\> + \ldots + \<m-m', u_s \> = \<m-m', u_\tau \> = \<m, u_\tau \>,
$
by using the fact that
$E_{m-m'}^{\tau} = E^{\rho_1}(\< m-m', \, u_1 \>) \cap \ldots \cap
E^{\rho_s}(\< m-m', \, u_s \>)$
and
$\widetilde{E}_{m}^{\rho} = \widetilde{E}^{\rho}(\< m, \, u_{\tau} \>)$,
we get the result.
\end{proof}

The following example shows that the reflexive pullback of $\cE_F$ might not be saturated
in $\cE'$ in general. Hence, our hypothesis on $\cE$ being sufficiently smooth, or on
pulled back subsheaves being saturated, are necessary in the statements of our results.
\begin{example}
\label{ex:nonsaturated}
Let $(u_1, u_2)$ be a basis of $\Z^2$ and
$\Sigma = \{\Cone(A): A \subseteq \{u_1, u_2\} \}$. Let $\cE$ be an equivariant reflexive
sheaf on $X = \C^2$ given by the family of filtrations
$$
E^{\rho_1}(j) = \left\lbrace
\begin{array}{ll}
\{0\} & \text{if} ~ j \leq 0 \\
E_1 & \text{if} ~ 1 \leq j \leq 2 \\
E & \text{if} ~ j \geq 3
\end{array}
\right.
~ \text{and} ~
E^{\rho_2}(j) = \left\lbrace
\begin{array}{ll}
\{0\} & \text{if} ~ j \leq 0 \\
E_2 & \text{if} ~ j=1 \\
E & \text{if} ~ j \geq 2
\end{array}
\right.
$$
where $E_1 = \Span(u_1)$, $E_2 = \Span(u_2)$ and $E = \Span(u_1, u_2)$. We denote by
$\pi: X' \rightarrow X$ the blowup along $V(\Cone(u_1,u_2))$ (that is the blowup at the
origin). We set $F = \Span(u_1 + u_2)$ and $\cE_F$ the subsheaf of $\cE$ given by
$F^{\rho}(j) = E^{\rho}(j) \cap F$. According to Proposition \ref{prop:filt-toric-blowup},
$$
\widetilde{E}^{\rho_0}(j) = \left\lbrace
\begin{array}{ll}
\{0\} & \text{if} ~ j \leq 2 \\
E_1 & \text{if} ~ j = 3 \\
E & \text{if} ~ j \geq 4
\end{array}
\right.
~ \text{and} ~
\widetilde{F}^{\rho_0}(j) = \left\lbrace
\begin{array}{ll}
\{0\} & \text{if} ~ j \leq 4 \\
F & \text{if} ~ j \geq 5
\end{array}
\right. ~.
$$
As $\widetilde{F}^{\rho_0}(4) \neq \widetilde{E}^{\rho_0}(4) \cap F$, we deduce that
$(\pi^\ast \cE_F)^{\vee \vee}$ is not saturated in $(\pi^\ast \cE)^{\vee \vee}$.
\end{example}

Let $D = \sum_{\rho \in \Sigma(1)}{a_{\rho} D_{\rho}}$ be a Cartier divisor of $X$.
According to \cite[Proposition 6.2.7]{CLS}, if $\varphi_D: |\Sigma| \rightarrow \R$ is the
support function of $D$, then the support function $\varphi_{D'}: |\Sigma'| \rightarrow \R$
of $D' = \pi^* D$ is given by $\varphi_{D'} = \varphi_D \circ \phi_\R$.
Therefore, by Proposition \ref{prop:support-function-divisor}, we have
\begin{equation}\label{eq:blowup-polarization}
\pi^{\ast} D = \sum_{\rho \in \Sigma(1)}{a_{\rho} D_{\rho}} +
\sum_{\rho \in \tau(1)} a_{\rho} D_0 ~.
\end{equation}
As $c_1(\cE') = \pi^\ast c_1(\cE)$, we get
\begin{equation}\label{eq:c1-sheaf-blowup}
c_1(\cE') = - \sum_{\rho \in \Sigma(1)} \iota_{\rho}( \cE) D_{\rho} -
\sum_{\rho \in \tau(1)} \iota_{\rho}( \cE) D_0 ~.
\end{equation}
In the following Lemma we give the expression of $c_1(\cE'_F)$ with respect to $c_1(\cE_F)$.

\begin{lem}\label{lem:c1-subsheaf-blowup}
Let $F$ be a vector subspace of $E$. The first Chern class of $\cE_{F}'$ is given by
$$
c_1(\cE_{F}') = \pi^{\ast} c_1(\cE_F) + \sum_{j \in \Z}{d_j(F) \, D_0}
$$
where
$d_j(F) = \dim (F \cap \widetilde{E}^{\rho_0}(j)) - \dim \widetilde{F}^{\rho_0}(j)$.
\end{lem}

\begin{proof}
By Corollary \ref{cor:pullback-sheaf-surjection}, if $\rho \in \Sigma(1)$, we have
$F \cap \widetilde{E}^{\rho}(j) = F^{\rho}(j)$. Thus,
for any $\rho \in \Sigma(1)$, $\iota_{\rho}( \cE_{F}') = \iota_{\rho}( \cE_F)$.
We now consider the case $\rho = \rho_0$. We have
$$
\iota_{\rho_0}( \cE_{F}')= \sum_{j \in \Z} j \left(
\dim(F \cap \widetilde{E}^{\rho_0}(j)) - \dim(F \cap \widetilde{E}^{\rho_0}(j-1)) \right)
$$
and
\begin{align*}
\iota_{\rho_0}( \cE_{F}') - \iota_{\rho_0}( \pi^{\ast} \cE_F) = &
\sum_{j \in \Z} j \left( \dim(F \cap \widetilde{E}^{\rho_0}(j)) -
\dim \widetilde{F}^{\rho_0}(j) \right)
\\ & -
\sum_{j \in \Z} j \left( \dim(F \cap \widetilde{E}^{\rho_0}(j-1)) -
\dim \widetilde{F}^{\rho_0}(j-1) \right)
\\ = &
\sum_{j \in \Z}{j \, d_j(F)} - \sum_{j \in \Z}{j \, d_{j-1}(F)}
\end{align*}
There are $p, \, q \in \Z$ with $p<q$ such that $d_j(F)= 0$ if $j<p$ and $d_j(F)= 0$
if $j>q$. Hence,
$$
\iota_{\rho_0}( \cE_{F}') - \iota_{\rho_0}( \pi^{\ast} \cE_F) =
\sum_{j=p}^{q}{j \, d_j(F)} - \sum_{j=p+1}^{q+1}{j \, d_{j-1}(F)} =
- \sum_{j \in \Z}{d_j(F)}
$$
and $c_1(\cE_{F}')= c_1(\pi^{\ast} \cE_F) + \left(\sum_{j \in \Z}{d_j(F)} \right) D_0$.
\end{proof}
With the results of these last two sections, we are ready to prove the various statements from the introduction related to blowing-up (semi)stable sheaves.

\subsection{Blowup in several points}\label{sec:blowup-along-point}
In this section, we give the proof of Theorem \ref{theo:blowuppoint}.
Let $S$ be a set of fixed-points of $X$ under the torus action and $\pi: X' \rightarrow X$ be
the blowup of $X$ in $S$. For any $p \in S$ there is $\sigma \in \Sigma(n)$ such that
$p= \gamma_\sigma$. We set $S_{\Sigma} = \{\sigma \in \Sigma(n): \gamma_\sigma \in S \}$.
According to Section \ref{sec:blowup-of-variety}, the fan $\Sigma'$ of
$X'$ is given by
$$
\Sigma' = \{ \sigma \in \Sigma : \sigma \notin S_\Sigma \} \cup \bigcup_{\sigma \in S_\Sigma}
\Sigma_{\sigma}^{\ast}(\sigma).
$$
For any $p \in S$, let $D_p \subseteq X'$ be the exceptional divisor over $p$, and
let $D_S = \sum_{p \in S}D_p$ be the total exceptional divisor. If $L$ is an ample divisor of
$X$ and $\ep >0$, we set $L_\ep = \pi^\ast L - \ep D_S$. We first observe that:

\begin{lem}
If $p, q \in S$ are distinct, then $[D_p] \cdot [D_q] = 0 \in A_{n-2}(X')$.
\end{lem}

\begin{proof}
Let $\sigma=\Cone(u_1, \ldots, u_n), \sigma'=\Cone(u'_1, \ldots, u'_n) \in \Sigma(n)$ such
that $p = \gamma_\sigma$ and $q=\gamma_{\sigma'}$. We set $u_\sigma= u_1 + \ldots + u_n$ and
$u_{\sigma'} = u'_1 + \ldots + u'_n$. The divisor $D_p$ (resp. $D_q$) corresponds
to the ray $\Cone(u_\sigma)$ (resp. $\Cone(u_{\sigma'})$). As
$\Cone(u_\sigma, u_{\sigma'}) \notin \Sigma_{\sigma}^{\ast}(\sigma) \cup
\Sigma_{\sigma'}^{\ast}(\sigma')$, by Lemma \ref{lem:intersection-toric-simplicial} we deduce
that $[D_p] \cdot [D_q] = 0 \in A_{n-2}(X')$.
\end{proof}

Let $\sigma = \Cone(u_1, \ldots, \, u_n) \in S_{\Sigma}$ and $p=\gamma_\sigma$. We denote by
$(e_1, \ldots, \, e_n)$ the dual basis of $(u_1, \ldots, \, u_n)$ and we set
$\rho_i = \Cone(u_i)$. We compute the intersection product on $X'$. We have
$[D_{\rho}] \cdot [D_p] = 0$ if $\rho \in \Sigma'(1) \setminus (\sigma(1) \cup \{ \Cone(u_\sigma) \})$.
For $i \in \{1, \ldots, \, n \}$, if we set $m = -e_i$,
by Lemma \ref{lem:intersection-toric-simplicial} we get
$$
[D_p] \cdot [D_p] = [D_p + \ddiv(\chi^m)] \cdot [D_p] = -[D_{\rho_i}] \cdot [D_p] ;
$$
therefore
$$
\left\lbrace
\begin{array}{ll}
D_{p}^n = (-1)^{n-1} & \\
D_{\rho} \cdot D_{p}^{n-1} = (-1)^n & \text{if}~ \rho \in \sigma(1) \\
D_{\rho} \cdot D_{p}^{n-1} = 0 & \text{if} ~ \rho \in \Sigma'(1) \setminus (\sigma(1) \cup \{ \Cone(u_\sigma) \})
\end{array}
\right..
$$
If $L =\sum_{\rho \in \Sigma(1)}a_\rho D_\rho$, then
$$
\pi^\ast L = \sum_{\rho \in \Sigma(1)}a_\rho D_\rho + \sum_{\sigma \in S_\Sigma}
\sum_{\rho \in \sigma(1)}a_{\rho} D_{\gamma_\sigma}~;
$$
hence, for any $p \in S$, $[\pi^\ast L] \cdot [D_p] = 0 \in A_{n-2}(X')$. Thus,
$$
L_{\ep}^{n-1} = (\pi^\ast L)^{n-1} + (-1)^{n-1} \ep^{n-1} \sum_{p \in S} D_{p}^{n-1} ~.
$$

\begin{proof}[Proof of Theorem \ref{theo:blowuppoint}]
For any $p \in S$, we have $\deg_{L_{\ep}}(D_p) = \ep^{n-1}$. If $\rho \in \Sigma(1)$, then
$$
\deg_{L_{\ep}}(D_{\rho})= \deg_{L}(D_{\rho}) - \sum_{\sigma \in S_\Sigma,\, \rho \in \sigma(1)} \ep^{n-1}.
$$
Thus,
\begin{align*}
\rk(\cE') \mu_{L_{\ep}}(\cE') = &
-\sum_{\rho \in \Sigma(1)} \iota_{\rho}( \cE) \deg_{L_{\ep}}(D_{\rho})-
\sum_{\sigma \in S_\Sigma} \sum_{\rho \in \sigma(1)} \iota_{\rho}( \cE) \deg_{L_{\ep}}(D_{\gamma_\sigma})
\\ = &
-\sum_{\rho \in \Sigma(1)} \iota_{\rho}( \cE) \deg_{L}(D_{\rho})
- \sum_{\sigma \in S_\Sigma} \sum_{\rho \in \sigma(1)} \iota_{\rho}( \cE) \ep^{n-1}
\\ & +
\sum_{\rho \in \Sigma(1)} \iota_{\rho}( \cE) \left( \sum_{\sigma \in S_\Sigma,\, \rho \in \sigma(1)} \ep^{n-1}\right)
\\ = & \rk(\cE) \mu_{L}(\cE).
\end{align*}
Hence, $\mu_{L_\ep}(\cE') = \mu_{L}(\cE)$. If $F$ is a vector subspace of $E$, the
same computation gives $\mu_{L_\ep}((\pi^\ast \cE_F)^{\vee \vee}) = \mu_{L}(\cE_F)$.
According to Lemma \ref{lem:c1-subsheaf-blowup}, for any vector subspace $F$ of $E$ one
has
$$
\mu_{L_\ep}(\cE')- \mu_{L_\ep}(\cE'_F) = \mu_{L}(\cE) - \mu_{L}(\cE_F) -
\dfrac{\ep^{n-1}}{\rk(\cF)} \sum_{p \in S} \sum_{j \in \Z}
(\dim (F \cap \widetilde{E}^{\rho_p}(j)) - \dim \widetilde{F}^{\rho_p}(j))
$$
where $\rho_p$ is the ray corresponding to the divisor $D_p$.
Hence, Theorem \ref{theo:blowuppoint} follows from Lemma \ref{lem:saturated-subsheaf}.
\end{proof}

\subsection{Blowup along a curve}\label{sec:blowup-along-cure}
In this section, we assume that $n=\dim(X)\geq 3$ and that $\tau \in \Sigma(n-1)$ is the
intersection of two $n$-dimensional cones $\sigma$ and $\sigma'$. Hence we consider the blowup $\pi : X' \to X$ along the curve $Z = V(\tau)$.
With the results of Section \ref{sec:slope-pullback-blowup}, we can give the proof of Theorem
\ref{theo:blowupcurve}.

\begin{proof}[Proof of Theorem \ref{theo:blowupcurve}]
Let $\cE$ be a strictly semistable sheaf on $(X, \, L)$.
According to Corollary \ref{cor:slope-pullback-blowup} one has
$$
\mu_{L_\ep}(\cE') = \mu_{L}(\cE) - \dfrac{\ep^{n-1}}{\rk(\cE)} c_1(\cE)\cdot V(\tau)
$$
and by (\ref{eq:degree-exceptional-div}) we have
$$
\deg_{L_\ep}(D_0) = (n-1)\ep^{n-2} L \cdot V(\tau) -(-1)^n \ep^{n-1} D_{0}^n ~.
$$
By Lemma \ref{lem:c1-subsheaf-blowup}, for any vector subspace $F$ of $E$, we have
\begin{multline}\label{eq:stability-blowup-curve}
\mu_{L_\ep}(\cE') - \mu_{L_\ep}(\cE'_F) = \mu_L(\cE) - \mu_L(\cE_F)
+ \ep^{n-1} \left( \dfrac{c_1(\cE_F) \cdot V(\tau)}{\rk(\cE_F)}
- \dfrac{c_1(\cE) \cdot V(\tau)}{\rk(\cE)} \right)
\\
- \dfrac{\ep^{n-2}}{\rk(\cE_F)} \left( (n-1) L \cdot V(\tau) -(-1)^n \ep \, D_{0}^n \right)
\sum_{j \in \Z}{d_j(F)}.
\end{multline}
Let $\cE_F \in \fE$ for $F \subsetneq E$. We set
$C = \sum_{j}{d_j(F)}$.
If $(\pi^* \cE_F)^{\vee \vee}$ is not saturated in $\cE'$, there is $\ep_F >0$ such that
for any $\ep \in (0,\ep_F)\cap \Q$,
$$
\ep \left( \dfrac{c_1(\cE_F) \cdot V(\tau)}{\rk(\cE_F)} -
\dfrac{c_1(\cE) \cdot V(\tau)}{\rk(\cE)} + \dfrac{(-1)^n C \times D_{0}^n}{\rk(\cE_F)}
\right) < \dfrac{(n-1)C \times L \cdot V(\tau)}{\rk(\cE_F)} ~.
$$
If $(\pi^* \cE_F)^{\vee \vee}$ is saturated in $\cE'$, then for $0 < \ep \ll 1$,
$\mu_{L_\ep}(\cE') - \mu_{L_\ep}(\cE'_F) > 0 \, (\text{resp}. \, \geq 0)$ if and only if
$$
\dfrac{c_1(\cE_F) \cdot V(\tau)}{\rk(\cE_F)} - \dfrac{c_1(\cE) \cdot V(\tau)}{\rk(\cE)}
> 0 \, (\text{resp.} \, \geq 0) ~.
$$
With these two observations about $(\pi^* \cE_F)^{\vee \vee}$ for $\cE_F \in \fE$,
we deduce the result.
\end{proof}

It should be clear now that the proof of Theorem \ref{theo:blowuphigher} follows as for the one of Theorem \ref{theo:blowupcurve}, by mean of Formula (\ref{eq:formulaintro}). We leave the details to the interested reader. We turn now to an explicit formula that helps applying Theorem \ref{theo:blowupcurve} on concrete examples. For a divisor $D$ of $X$, we can compute $D \cdot V(\tau)$ by using the fact that
$\tau = \sigma \cap \sigma'$. Let $\Sigma_0 = \sigma(1) \cup \sigma'(1)$. There is a family of numbers $\alpha_{\rho} \in \Z$ such that
$$
\sum_{\rho \in \Sigma_0} \alpha_{\rho} u_{\rho} = 0 \quad \text{and} \quad
\alpha_{\rho} = 1 ~ \text{if} ~ \rho \in \Sigma_0 \setminus \tau(1) ~.
$$
We assume that $\sigma = \Cone(u_1, \ldots, \, u_n)$,
$\sigma' = \Cone(u_1, \ldots, \, u_{n-1}, \, u_{n+1})$ and
$\Sigma_0 = \{ \Cone(u_i) : 1 \leq i \leq n+1 \}$.
For $i \in \{1, \ldots, \, n+1\}$, we set $\rho_i = \Cone(u_i)$ and
$\alpha_i = \alpha_{\rho_i}$.
We denote by $(e_1, \ldots, \, e_n)$ the dual basis of $(u_1, \ldots, \, u_n)$.
For $i \in \{1, \ldots, \, n-1\}$, we have
$$
D_{\rho_i} \sim_{\rm lin} D_{\rho_i} + \ddiv(\chi^{-e_i}) = \alpha_i \, D_{\rho_{n+1}} +
\sum_{\rho \in \Sigma(1) \setminus \Sigma_0}{\<-e_i, \, u_\rho \> D_{\rho}} ~.
$$
By Lemma \ref{lem:intersection-toric-simplicial}, we get
$$
D_\rho \cdot V(\tau) =
\left\lbrace
\begin{array}{ll}
\alpha_\rho & \text{if}~ \rho \in \Sigma_0 \\
0 & \text{if}~ \rho \in \Sigma(1)\setminus \Sigma_0
\end{array}
\right. .
$$
Hence,
\begin{equation}\label{eq:curve-condition-stability}
\dfrac{c_1(\cE_F) \cdot V(\tau)}{\rk(\cE_F)} - \dfrac{c_1(\cE) \cdot V(\tau)}{\rk(\cE)}=
\sum_{\rho \in \Sigma_0} \alpha_{\rho} \left( \dfrac{\iota_{\rho}( \cE)}{\rk(\cE)} -
\dfrac{\iota_{\rho}( \cE_F)}{\rk(\cE_F)} \right) ~.
\end{equation}

\subsection{Examples of (de)stabilizing blow-ups along curves}
\label{sec:example-picard-rank2}
Let $X=X_\Sigma$ be a smooth toric variety of dimension $n$ given by
$$
X = \P \left( \cO_{\P^1}^{\oplus r} \oplus \cO_{\P^1}(1) \right)
$$
with $r \geq 2$ such that $r + 1 = n$.
We denote by $\pr : X \rightarrow \P^1$ the projection to the base $\P^1$.
By \cite[Section 7.3]{CLS}, the rays of $\Sigma$ are given by the half-lines generated by
$w_0, w_1, v_0, v_1, \ldots, v_r$ where $(w_1, v_1, \ldots, v_r)$ is the standard basis of
$\Z^{r+1}$,
$$
v_0 = -(v_1 + \ldots + v_r) \quad \text{and} \quad w_0 = v_r -w_1 ~.
$$
The maximal cones of $\Sigma$ are given by
$$
\Cone(w_j) + \Cone(v_0, \ldots, \,\widehat{v_i}, \ldots, \, v_r)
$$
where $j \in \{0, 1\}$ and $i \in \{0, \ldots, \, r \}$.
We denote by $D_{v_i}$ the divisor corresponding to the ray $\Cone(v_i)$ and $D_{w_j}$
the divisor corresponding to the ray $\Cone(w_j)$. If $\nu \in \Q_{+}^{\ast}$, then
$\pr^{\ast} \cO_{\P^1}(\nu) \otimes \cO_{X}(1) \cong \cO_{X}(\nu D_{w_0} + D_{v_0})$
is a rational polarization of $X$.
Let $\cE$ be the tangent sheaf of $X$. The family of filtrations of $\cE$ is given in
Example \ref{examp:filtration-tangent}.
According to \cite[Theorem 1.4]{HNS19}, the sheaf $\cE$ is stable with respect to
$L = \pr^{\ast} \cO_{\P^1}(\nu) \otimes \cO_{X}(1)$ if and only if
$0 < \nu < \nu_0$ with $\nu_0 = \frac{1}{r+1}$.

We now assume that
$L = \pr^{\ast} \cO_{\P^1}(1/(r+1)) \otimes \cO_{X}(1)$.
The sheaf $\cE$ is strictly semistable with respect to $L$. The subsheaf $\cE_F$ with
$F = \Span(v_0, \ldots, \, v_r)$ is the unique saturated subsheaf of $\cE$ such
that $\mu_{L}(\cE_{F}) = \mu_{L}(\cE)$.
The family of filtrations of $\cE_F$ are given by
$$
F^{\rho}(j) = \left\lbrace
\begin{array}{ll}
0 & \text{if}~ j< -1 \\
\Span(u_{\rho}) & \text{if}~ j = -1 \\
F & \text{if} ~ j> -1
\end{array}
\right. \quad \text{if}~ \rho = \Cone(v_i)
$$
and by
$$
F^{\rho}(j) = \left\lbrace
\begin{array}{ll}
0 & \text{if}~ j< 0 \\
F & \text{if} ~ j \geq 0
\end{array}
\right. \quad \text{if} ~ \rho = \Cone(w_j) ~.
$$
Hence,
$$
\dfrac{\iota_{\rho}( \cE)}{r+1} - \dfrac{\iota_{\rho}( \cE_F)}{r} = \left\lbrace
\begin{array}{ll}
\frac{1}{r} - \frac{1}{r+1} & \text{if}~ \rho = \Cone(v_i) \\
\frac{-1}{r+1} & \text{if}~ \rho = \Cone(w_j)
\end{array}
\right. ~.
$$
Given $\tau \in \Sigma(n-1)$, in the following examples, we study the stability of the
reflexive pullback $\cE' = (\pi^{\ast} \cE)^{\vee \vee}$ on $X' = \Bl_{V(\tau)}(X)$
with respect to small perturbations of $\pi^{\ast}L$. In these examples,
$(\pi^\ast \cE_F)^{\vee \vee}$ is saturated in $\cE'$.

\begin{example}
Let $\tau = \Cone(w_0, \, v_1, \ldots, v_{r-1})$. We have
$$
\tau = \Cone(w_0, \, v_1, \ldots, v_{r-1}, \, v_r) \cap
\Cone(w_0, \, v_1, \ldots, v_{r-1}, \, v_0) ~.
$$
As $0 \cdot w_0 + v_0 + v_1 + \ldots + v_r = 0$, by Equation
(\ref{eq:curve-condition-stability}) we get
$$
\dfrac{c_1(\cE_F) \cdot V(\tau)}{r} - \dfrac{c_1(\cE) \cdot V(\tau)}{r+1} =
\dfrac{r+1}{r} - 1
$$
So there is $\ep_0 > 0$ such that for any $\ep \in (0,\ep_0)\cap \Q$, $\cE'$ is stable with respect to $L_\ep$.
\end{example}

\begin{example}
Let $\tau = \Cone(v_0, \, v_1, \ldots, v_{r-1})$. We have
$$
\tau = \Cone(v_0, \, v_1, \ldots, v_{r-1}, \, w_0) \cap
\Cone(v_0, \, v_1, \ldots, v_{r-1}, \, w_1)
$$
As $w_0 + w_1 + v_0 + v_1 + \ldots + v_{r-1} = 0$, by Equation
(\ref{eq:curve-condition-stability}) we get
$$
\dfrac{c_1(\cE_F) \cdot V(\tau)}{r} - \dfrac{c_1(\cE) \cdot V(\tau)}{r+1}=
\dfrac{-1}{r+1} ~.
$$
Hence, there is $\ep_0 > 0$ such that for any $\ep \in (0,\ep_0)\cap \Q$, $\cE'$
is unstable with respect to $L_\ep$.
\end{example}

\bibliography{toricpullbacks}

\providecommand{\bysame}{\leavevmode\hbox to3em{\hrulefill}\thinspace}
\providecommand{\MR}{\relax\ifhmode\unskip\space\fi MR }
\providecommand{\MRhref}[2]{%
  \href{http://www.ams.org/mathscinet-getitem?mr=#1}{#2}
}
\providecommand{\href}[2]{#2}
\begin{thebibliography}{10}

\bibitem{BS}
Shigetoshi Bando and Yum-Tong Siu, \emph{{Stable sheaves and Einstein-Hermitian
  metrics}}, Geometry and analysis on complex manifolds. Festschrift for
  Professor S. Kobayashi's 60th birthday (1994), 39--59.

\bibitem{Buch}
Nicholas~P. Buchdahl, \emph{{Blowups and gauge fields.}}, Pac. J. Math.
  \textbf{196} (2000), no.~1, 69--111.

\bibitem{CT22}
Andrew Clarke and Carl Tipler, \emph{{Equivariant stable sheaves and toric
  GIT}}, Proceedings of the Royal Society of Edinburgh: Section A Mathematics
  (2022), 1--32.

\bibitem{CLS}
David Cox, John Little, and Hal Schenck, \emph{{Toric varieties}}, Graduate
  studies in mathematics, American Mathematical Soc., 2011.

\bibitem{DDK20b}
Jyoti Dasgupta, Arijit Dey, and Bivas Khan, \emph{Erratum for {``Stability of
  equivariant vector bundle over toric varieties''}}.

\bibitem{Cataldo-toricmaps}
Mark~Andrea de~Cataldo, Luca Migliorini, and Mircea Musta{\c{t}}{\u{a}},
  \emph{{Combinatorics and topology of proper toric maps}}, Journal für die
  reine und angewandte Mathematik (Crelles Journal) \textbf{2018} (2016),
  no.~744, 133--163.

\bibitem{DerSek}
Ruadha{\'{\i}} Dervan and Lars~Martin Sektnan, \emph{{Hermitian Yang-Mills
  connections on blowups}}, J. Geom. Anal. \textbf{31} (2021), no.~1, 516--542.

\bibitem{Fulton-intersection}
William Fulton, \emph{{Intersection theory}}, Springer New York, 1984.

\bibitem{Ful93}
\bysame, \emph{{Introduction to toric varieties}}, Annals of Mathematics
  Studies, Princeton Univ. Press, Princeton, NJ, 1993.

\bibitem{Graf-MR-for-linear-system}
Patrick Graf, \emph{{A Mehta Ramanathan theorem for linear systems with
  basepoints}}, Mathematische Nachrichten \textbf{289} (2016), no.~10,
  1208--1218.

\bibitem{GKPeternell15}
Daniel Greb, Stefan Kebekus, and Thomas Peternell, \emph{{Movable curves and
  semistable sheaves}}, International Mathematics Research Notices
  \textbf{2016} (2015), no.~2, 536--570.

\bibitem{Har80}
Robin Hartshorne, \emph{{Stable reflexive sheaves}}, Math. Annalen \textbf{254}
  (1980), 121--176.

\bibitem{HNS19}
Milena Hering, Benjamin Nill, and Hendrik Suess, \emph{{Stability of tangent
  bundles on smooth toric Picard-rank-2 varieties and surfaces}}, London
  Mathematical Society Lecture Note Series 473, pp.~1--25, Cambridge University
  Press, 2022.

\bibitem{HuLe}
Daniel Huybrechts and Manfred Lehn, \emph{{The geometry of moduli spaces of
  sheaves}}, 2 ed., Cambridge Mathematical Library, Cambridge University Press,
  2010.

\bibitem{Kly90}
Alexander Klyachko, \emph{{Equivariant bundle on toral varieties}}, Mathematics
  of the USSR-Izvestiya \textbf{35} (1990), no.~2, 337--375.

\bibitem{Koo11}
Martijn Kool, \emph{{Fixed point loci of moduli spaces of sheaves on toric
  varieties}}, Advances in Mathematics \textbf{227} (2011), no.~4, 1700--1755.

\bibitem{Laz-positivity1}
Robert Lazarsfeld, \emph{{Positivity in algebraic geometry I}}, Ergebnisse der
  Mathematik und ihrer Grenzgebiete. 3. Folge A Series of Modern Surveys in
  Mathematics, Springer, 2004.

\bibitem{MeRa}
Vikram~B. Mehta and Annamalai Ramanathan, \emph{{Restriction of stable sheaves
  and representations of the fundamental group}}, Invent. Math. \textbf{77}
  (1984), no.~1, 163--172.

\bibitem{NapTip}
Achim Napame and Carl Tipler, \emph{{On the singular locus of toric sheaves}},
  In preparation (2022).

\bibitem{payne2008}
Sam Payne, \emph{{Moduli of toric vector bundles}}, Compositio Mathematica
  \textbf{144} (2008), no.~5, 1199--1213.

\bibitem{Per03}
Markus Perling, \emph{{Resolutions and moduli for equivariant sheaves over
  toric varieties}}, Ph.D. thesis, University of Kaiserslautern, 2003.

\bibitem{Per04}
\bysame, \emph{{Graded rings and equivariant sheaves on toric varieties}},
  Mathematische Nachrichten \textbf{263/264} (2004), 181--197.

\bibitem{SekTip22}
Lars~Martin Sektnan and Carl Tipler, \emph{{Hermitian Yang--Mills connections
  on pullback bundles}}, Preprint ArXiv 2006.06453 (2022).

\bibitem{Sib}
Benjamin Sibley, \emph{{Asymptotics of the {Yang}-{Mills} flow for holomorphic
  vector bundles over {K{\"a}hler} manifolds: the canonical structure of the
  limit}}, J. Reine Angew. Math. \textbf{706} (2015), 123--191.

\bibitem{Tak72}
Fumio Takemoto, \emph{{Stable vector bundles on algebraic surfaces}}, Nagoya
  Math. J. \textbf{47} (1972), 29--48.

\end{thebibliography}
\bibliographystyle{amsplain}

\end{document}